\documentclass[a4paper, 10pt, final]{amsart}

\usepackage{xcolor}

\usepackage[T1]{fontenc}
\usepackage[english]{babel}
\usepackage{fullpage}
\usepackage{amsmath}
\usepackage{dsfont}\let\mathbb\mathds
\usepackage{setspace}
\usepackage{amsfonts}
\usepackage{amssymb}
\usepackage{graphicx}
\usepackage{amscd}
\usepackage{mathrsfs}
\usepackage{fancybox}

\usepackage{lmodern}
\usepackage{graphicx}
\usepackage{fancyheadings}
\usepackage{chngcntr}
\counterwithout{figure}{section}
\usepackage{tikz}
\usetikzlibrary{calc}
\usepackage{pgfplots}
\usepgfplotslibrary{groupplots}
\usepgfplotslibrary{units}

\usetikzlibrary{arrows}
\usepackage{tkz-berge}

\usetikzlibrary{graphs,shapes.geometric,arrows,fit,matrix,positioning}

\usepackage[extra]{tipa}

\usepackage[OT2,T1]{fontenc}

\usepackage{amsmath, amsthm, amsfonts}
\cornersize{2}
\makeatletter
\newcounter {subsubsubsection}[subsubsection]
\renewcommand\thesubsubsubsection{\thesubsubsection .\@alph\c@subsubsubsection}
\newcommand\subsubsubsection{\@startsection{subsubsubsection}{4}{\z@}% 
                                     {-3.25ex\@plus -1ex \@minus -.2ex}%
                                     {1.5ex \@plus .2ex}%
                                     {\normalfont\normalsize\bfseries}}
\newcommand*\l@subsubsubsection{\@dottedtocline{4}{10.0em}{4.1em}}
\newcommand*{\subsubsubsectionmark}[1]{}
\makeatother

\usepackage{graphicx}

\DeclareMathOperator{\Lie}{Lie}
\DeclareMathOperator{\weight}{weight}
\DeclareMathOperator{\id}{id}

\DeclareMathOperator{\loc}{loc}

\DeclareMathOperator{\Ad}{Ad}
\DeclareMathOperator{\coeff}{coeff}
\DeclareMathOperator{\Frac}{Frac}

\DeclareMathOperator{\crys}{crys}
\DeclareMathOperator{\depth}{depth}

\DeclareMathOperator{\DR}{DR}
\DeclareMathOperator{\rev}{rev}
\DeclareMathOperator{\iter}{iter}

\DeclareMathOperator{\Spec}{Spec}

\DeclareMathOperator{\DMR}{DMR}

\DeclareMathOperator{\RT}{RT}
\DeclareMathOperator{\inv}{inv}
\DeclareMathOperator{\Gal}{Gal}

\DeclareMathOperator{\dch}{dch}

\DeclareMathOperator{\Orb}{Orb}

\DeclareMathOperator{\un}{un}
\DeclareMathOperator{\Li}{Li}

\DeclareMathOperator{\elim}{elim}
\DeclareMathOperator{\Vect}{Vect}
\DeclareMathOperator{\Map}{Map}

\DeclareMathOperator{\shft}{shft}
\DeclareMathOperator{\comp}{comp}
\DeclareMathOperator{\har}{har}

\DeclareMathOperator{\KZ}{KZ}

\DeclareMathOperator{\an}{an}

\theoremstyle{definition}

\newtheorem{Theorem}{Theorem}[section]
\newtheorem{Proposition}[Theorem]{Proposition}

\newtheorem{Lemma}[Theorem]{Lemma}

\newtheorem{Definition}[Theorem]{Definition}

\newtheorem{Interpretation}[Theorem]{Interpretation}
\newtheorem{Fact}[Theorem]{Fact}

\newtheorem{Example}[Theorem]{Example}

\newtheorem{Lemma-Definition}[Theorem]{Lemma-Definition}

\newtheorem{Comment}[Theorem]{Comment}

\newtheorem{Remark}[Theorem]{Remark}
\newtheorem{Proposition-Definition}[Theorem]{Proposition-Definition}
\newtheorem{Nota Bene}[Theorem]{Nota Bene}

\newtheorem{Examples}[Theorem]{Examples}

\input cyracc.def
\DeclareFontFamily{U}{russian}{}
\DeclareFontShape{U}{russian}{m}{n}
        { <5><6> wncyr5
        <7><8><9> wncyr7
        <10><10.95><12><14.4><17.28><20.74><24.88> wncyr10 }{}
\DeclareSymbolFont{Russian}{U}{russian}{m}{n}
\DeclareSymbolFontAlphabet{\mathcyr}{Russian}
\makeatletter
\let\@math@cyr\mathcyr
\renewcommand{\mathcyr}[1]{\@math@cyr{\cyracc #1}}
\makeatother
\newcommand{\sh}{\mathcyr{sh}} % Le produit shuffle

\usepackage[top=3.5cm, bottom=3.5cm, left=2.5cm, right=2.5cm]{geometry}

\author{David Jarossay}

\title{}

\address{Institut de Recherche Math\'{e}matique Avanc\'{e}e, Universit\'{e} de Strasbourg, 7 rue Ren\'{e} Descartes, 67000 Strasbourg, France}
\email{jarossay@math.unistra.fr}

\begin{document}

\begin{center}
\begin{Large}
\textbf{AN EXPLICIT THEORY OF $\pi_{1}^{\un,\crys}(\mathbb{P}^{1} - \{0,\mu_{N},\infty\})$}\end{Large}
\\ \text{ }
\begin{large}
\\ \textbf{II : Algebraic relations of cyclotomic $p$-adic multiple zeta values
\\ \text{ }
\\ II-2 : From standard algebraic relations of weighted multiple harmonic sums to those of cyclotomic $p$-adic multiple zeta values}
\end{large}
\end{center}

\maketitle

\noindent
\newline

\begin{abstract}
Let $X_{0}=\mathbb{P}^{1} - (\{0,\infty\} \cup \mu_{N})\text{ }/\text{ }\mathbb{F}_{q}$, with $N \in \mathbb{N}^{\ast}$ and $\mathbb{F}_{q}$ of characteristic $p>0$ and containing a primitive $N$-th root of unity. We establish an explicit theory of the crystalline pro-unipotent fundamental groupoid of $X_{0}$.
\newline In part I, we have computed explicitly the Frobenius, and in particular cyclotomic $p$-adic multiple zeta values. In part II, we use part I to understand the algebraic relations of cyclotomic $p$-adic multiple zeta values via explicit formulas ; this is in particular a study of the harmonic Ihara actions and the maps of comparisons between them introduced in I-2 and I-3.
\newline In II-1, we have developed the basics of algebraic theory of cyclotomic sequences of prime weighted multiple harmonic sums and adjoint cyclotomic multiple zeta values viewed as variants of those of the algebraic theory of cyclotomic multiple zeta values.
\newline In this II-2, we use part I and II-1 to show that one can read some standard algebraic relations of cyclotomic $p$-adic multiple zeta values via the explicit formulas and  via the standard algebraic relations of sequences of multiple harmonic sums. This amounts to say that the harmonic Ihara actions and the comparison maps are compatible with algebraic relations. The two main results are two "harmonic" versions of Besser-Furusho-Jafari's theorem that $p$-adic multiple zeta values satisfy the regularized double shuffle relations. This gives two different answers to what we could call "the adjoint variant" of a question of Deligne and Goncharov about reading the quasi-shuffle relation of $p$-adic multiple zeta values via explicit formulas.
\end{abstract}

\noindent
\newline

\numberwithin{equation}{section}

\tableofcontents

\newpage

\section{Introduction}

\subsection{}

This is the second part of an explicit study of the crystalline pro-unipotent fundamental groupoid of $X_{0}=\mathbb{P}^{1} - (\{0,\infty\} \cup \mu_{N}) \text{ }/\text{ }\mathbb{F}_{q}$, where $N \in \mathbb{N}^{\ast}$ is prime to the characteristic $p$ of $\mathbb{F}_{q}$, and $\mathbb{F}_{q}$ contains a primitive $N$-th root of unity. In the first part, we have computed explicitly the Frobenius. In this part II, we use this computation to understand explicitly the algebraic relations between the periods associated with it : $p$-adic multiple zeta values ($N=1$) and their "cyclotomic" generalizations (any $N$).
\newline In this II-2, we show how to retrieve explicitly algebraic relations of $p$-adic multiple zeta values from the ones of multiple harmonic sums, using the formulas of part I, and using the part II-1 which was, from this point of view, a preliminary step.

\subsection{}

The basics of the algebraic theory of cyclotomic multiple zeta values are recalled in II-1, \S1. In this paper, we will focus in particular on double shuffle relations, which are among the standard algebraic properties of cyclotomic multiple zeta values. Let $\xi \in \overline{\mathbb{Q}} \hookrightarrow \mathbb{C}$ be a primitive $N$-th root of unity, $z_{i}= \xi^{i}$ for $i \in \{1,\ldots,N\}$, $z_{0} = 0$, and $\omega_{z_{i}} = \frac{dz}{z-z_{i}}$ for all $i \in \{0,\ldots,N\}$. Cyclotomic multiple zeta values are complex numbers which have an expression as iterated integrals and an expression as iterated series :
\begin{equation} \label{eq:integral} \zeta \big(
\begin{array}{c} z_{j_{d}},\ldots,z_{j_{1}} \\ s_{d},\ldots,s_{1} \end{array} \big) = \int_{0<t_{1}<\ldots<t_{n}<1} \omega_{z_{i_{n}}}(t_{n}) \ldots  \omega_{z_{i_{2}}}(t_{2}) \omega_{z_{i_{1}}}(t_{1}) = \sum_{0<n_{1}<\ldots <n_{d}}
\frac{\big( \frac{z_{j_{2}}}{z_{j_{1}}} \big)^{n_{1}} \ldots \big(\frac{1}{z_{j_{d}}}\big)^{n_{d}}}{n_{1}^{s_{1}}\ldots n_{d}^{s_{d}}} 
\end{equation}
\noindent where $j_{1},\ldots,j_{d} \in \{1,\ldots,N\}$, $s_{1},\ldots,s_{d} \in \mathbb{N}^{\ast}$, and $(i_{n},\ldots,i_{1})=(\underbrace{0, \ldots, 0}_{s_{d}-1}, j_{d}, \ldots \underbrace{0, \ldots ,0}_{s_{1}-1},j_{1})$, such that $(s_{d},j_{d}) \not= (1,N)$. The integer $s_{d}+\ldots+s_{1}$ is called the weight and the integer $d$ is called the depth.
\newline The double shuffle relations (i.e. the shuffle relation coming from iterated integrals and the quasi-shuffle relation coming from iterated series), are two ways of expressing any product $\zeta(w)\zeta(w')$, for two words $w,w'$, as a $\mathbb{Z}$-linear combination of cyclotomic multiple zeta values, implied respectively by their two expressions from equation (\ref{eq:integral}). The first examples, in low weight, resp. in low depth, follow from the following facts :
\begin{equation} \label{eq:exemple integral}\int_{0<t_{1}<1} \times \int_{0<t'_{1}<1} = 
\int_{0<t_{1}<t'_{1}<1} +\int_{0<t'_{1}<t_{1}<1} 
\end{equation}
\begin{equation} \label{eq:exemple series}\sum_{0<n_{1}} \times \sum_{0<n'_{1}} = 
\sum_{0<n_{1} < n'_{1}} + \sum_{0<n'_{1} <n_{1}} + \sum_{0<n_{1}=n'_{1}} 
\end{equation}
\noindent The general formulas, $\zeta(w)\zeta(w')=\zeta(w\text{ }\sh\text{ }w')$ and $\zeta(w)\zeta(w')=\zeta(w \ast w')$, where $\sh$ is the shuffle product and $\ast$ is the quasi-shuffle products, follow from the natural generalization of (\ref{eq:exemple integral}) and (\ref{eq:exemple series}) to all weights, resp. all depths.
\newline Let us consider the alphabet $e_{Z} = \{e_{0},e_{z_{1}},\ldots,e_{z_{N}}\}$, let us associate with each letter $e_{z_{j}}$ the differential form $\frac{dz}{z-z_{j}}$, and let us view the cyclotomic multiple zeta value in (\ref{eq:integral}) as indexed by the word $e_{0}^{s_{d}-1}e_{z_{j_{d}}} \ldots e_{0}^{s_{1}-1}e_{z_{j_{1}}}$. The generating series $\Phi_{\KZ}$ of cyclotomic multiple zeta values is an element of the non-commutative algebra of power series
$$ \mathbb{C} \langle \langle e_{Z} \rangle\rangle =
\mathbb{C} \langle \langle e_{0},e_{z_{1}},\ldots,e_{z_{N}} \rangle\rangle = \{ \sum_{w\text{ word on }e_{Z}} f[w] w \text{ }|\text{ }\forall w,\text{ } f[w] \in \mathbb{C} \} $$
\noindent and we have, for all words, $(-1)^{d}\Phi_{\KZ}[e_{0}^{s_{d}-1}e_{z_{j_{d}}} \ldots e_{0}^{s_{1}-1}e_{z_{j_{1}}}] = \zeta \big(
\begin{array}{c} z_{i_{d}},\ldots,z_{i_{1}} \\ s_{d},\ldots,s_{1} \end{array} \big)$.
\newline Let $K = \Frac(W(\mathbb{F}_{q})) \subset \overline{\mathbb{Q}_{p}}$ ; for each $\alpha \in \mathbb{N}^{\ast}$, one has an element $\Phi_{p,\alpha} \in K\langle \langle e_{Z}\rangle\rangle$ defined via the Frobenius of $\pi_{1}^{\un,\crys}(X_{0})$ iterated $\alpha$ times, whose coefficients are called the $p$-adic cyclotomic multiple zeta values defined as $\zeta_{p,\alpha}\big(
\begin{array}{c} z_{j_{d}},\ldots,z_{j_{1}} \\ s_{d},\ldots,s_{1} \end{array} \big) = (-1)^{d}\Phi_{p,\alpha}[e_{0}^{s_{d}-1}e_{z_{j_{d}}} \ldots e_{0}^{s_{1}-1}e_{z_{j_{1}}}]$. They satisfy the double shuffle relations \cite{Besser Furusho}, \cite{Furusho Jafari}. Conjecturally, the algebraic relations that they satisfy are those of cyclotomic multiple zeta values modulo the ideal $(\zeta(2))$.

\subsection{} Let us review how we have expressed the explicit computation of the Frobenius, more specifically, the computation of $\zeta_{p,\alpha}$, in part I.
\newline 
\newline The explicit formulas for cyclotomic $p$-adic multiple zeta values are expressed in terms of the weighted multiple harmonic sums, which are, essentially, the coefficients of the series expansion at $0$ of the canonical solution to the connexion $\nabla_{\KZ}$ on $\pi_{1}^{\un,\DR}(X_{K})$, called hyperlogarithms. They are the numbers

\begin{equation}
\label{eq:har}\har_{n} \big(
\begin{array}{c} z_{j_{d+1}},\ldots,z_{j_{1}} \\ s_{d},\ldots,s_{1} \end{array} \big) = n^{s_{d}+\ldots+s_{1}} \sum_{0<n_{1}<\ldots<n_{d}<n}
\frac{\big( \frac{z_{j_{2}}}{z_ {j_{1}}} \big)^{n_{1}} \ldots \big(\frac{z_{j_{d+1}}}{z_{j_{d}}}\big)^{n_{d}} 
\big(\frac{1}{z_{j_{d+1}}}\big)^{n}}{n_{1}^{s_{1}}\ldots n_{d}^{s_{d}}} \in \overline{\mathbb{Q}}
\end{equation}

\noindent with $n \in \mathbb{N}^{\ast}$, $s_{d},\ldots,s_{1} \in \mathbb{N}^{\ast}$ and $j_{d+1},\ldots,j_{1} \in \{1,\ldots,N\}$.
\newline We call prime weighted multiple harmonic sums the weighted multiple harmonic sums whose upper bound $n$ is a power of $p$.
\newline 
\newline We have defined three frameworks of computations, which will be formalized as operads in II-3. We denote them by DR, DR-RT and RT, where DR is a short for De Rham as usual, and RT is a short for "rational", "Rham-Taylor", and "rigid-Taylor" at the same time, and refers to computations on multiple harmonic sums viewed as $p$-adic numbers. The definition of these frameworks of computation is reviewed in \S1.3 of II-1.
\newline 
\newline We have defined a variant $K \langle \langle e_{Z}\rangle\rangle_{\har}$ of $K \langle \langle e_{Z}\rangle\rangle$, designed for containing the non-commutative generating series of multiple harmonic sums $\har_{n}=(\har_{n}(w))_{w\text{ word}}$, $n \in \mathbb{N}^{\ast}$. It arises in three different ways, which turn out to be isomorphic to each other : $K \langle \langle e_{Z}\rangle\rangle_{\har}^{\DR} \simeq K \langle \langle e_{Z}\rangle\rangle_{\har}^{\DR-\RT} \simeq K \langle \langle e_{Z}\rangle\rangle^{\RT}_{\har}$. For each $I \subset \mathbb{N}$, we view the sequence $\har_{I} = (\har_{n})_{n\in I} $ as an element of $\Map(I,K \langle \langle e_{Z}\rangle\rangle_{\har})$.
\newline 
\newline Using the differential equation of the Frobenius, we have defined by explicit formulas three continuous actions $\circ_{\har}^{\DR}$, $\circ_{\har}^{\DR-\RT}$ and $\circ_{\har}^{\RT}$ of certain topological groups, defined as subgroups of $\pi_{1}^{\un}(X_{K},-\vec{1}_{1},\Vect{1}_{0})(K)$ viewed canonically as a group, on, respectively, $\Map(\mathbb{N},K \langle \langle e_{Z}\rangle\rangle_{\har}^{\DR})$, $\Map(\mathbb{N},K \langle \langle e_{Z}\rangle\rangle_{\har}^{\DR-\RT})$ and
\newline $\Map(\mathbb{N},K \langle \langle e_{Z}\rangle\rangle^{\RT}_{\har})$. We call them harmonic Ihara actions, because of their relation with the Ihara product on $\pi_{1}^{\un}(X_{K},-\vec{1}_{1},\vec{1}_{0})$ and because they are used to express equations involving multiple harmonic sums and to study certain sequences of multiple harmonic sums as periods.
\newline 
\newline The equations from I-2 making $\zeta_{p,\alpha}$ explicit are (see \S2 and I-2 for the definition of each term) are :
\begin{equation} \label{eq:har DR RT}
\har_{p^{\alpha}\mathbb{N}} =\Ad_{\Phi_{p,\alpha}}(e_{1})
\text{ } \circ_{\har}^{\DR,\RT} \text{ }
\har_{\mathbb{N}}^{(p^{\alpha})}
\end{equation}
\begin{equation} \label{eq:har RT RT} 
\har_{p^{\alpha}\mathbb{N}} =  \har_{p^{\alpha}}
\text{ } \circ_{\har}^{\RT} \text{ }
\har_{\mathbb{N}}^{(p^{\alpha})} 
\end{equation}

\noindent Comparing equations (\ref{eq:har DR RT}) and (\ref{eq:har RT RT}) gave in I-2 the "comparison maps" $\Sigma^{\RT}$ and $\Sigma^{\DR}_{\inv}$ satisfying the following equations :

\begin{equation} \label{eq:series expansion}\Ad_{\Phi_{p,\alpha}}(e_{1}) = \Sigma^{\RT}\har_{p^{\alpha}} 
\end{equation}
\begin{equation} \label{eq:inverse series expansion}\har_{p^{\alpha}} = \Sigma^{\DR}_{\inv} \Ad_{\Phi_{p,\alpha}}(e_{1})
\end{equation}
\noindent And finally we proved in I-2 :
\begin{equation}
\Sigma_{\inv}^{\DR} \circ \Sigma^{\RT} = \id
\end{equation}

\noindent In this preliminary version of this paper, we omit the formulas from I-3, which involve the harmonic Ihara action $\circ_{\har}^{\DR}$ and express how the Frobenius iterated $\alpha$ times ($\alpha \in \mathbb{N}^{\ast}$) varies in function of $\alpha$ viewed as a $p$-adic integer.

\subsection{} The problematic of this II-2 is to read algebraic relations of cyclotomic $p$-adic multiple zeta values via explicit formulas.
\newline Deligne and Goncharov have asked this question for the quasi-shuffle relation (for $N=1$ and $\alpha=1$) in \cite{Deligne Goncharov}, \S5.28, just after their definition of $p$-adic multiple zeta values ; here is their formulation : "\emph{il serait int\'{e}ressant aussi de disposer pour ces coefficients d'expressions $p$-adiques qui rendent clair qu'ils v\'{e}rifient des identit\'{e}s du type}
$$ \coeff(e_{0}^{n-1}e_{1})\coeff(e_{0}^{m-1}e_{1})
= \coeff(e_{0}^{m-1}e_{1}e_{0}^{n-1}e_{1}) + \coeff(e_{0}^{m-1}e_{1}e_{0}^{n-1}e_{1}) + 
\coeff(e_{0}^{m+n-1}e_{1}) $$
\noindent \emph{qui pour $\dch(\sigma)$ expriment que}
$$ \sum \frac{1}{k^{n}} \sum \frac{1}{l^{m}} = 
\sum_{k>l} \frac{1}{k^{n}}\frac{1}{l^{m}} +  \sum_{l>k} \frac{1}{k^{n}}\frac{1}{l^{m}} + \sum \frac{1}{k^{n+m}} \text{ }^{"} $$

\noindent where $\dch(\sigma)$ defined in \cite{Deligne Goncharov}, \S5.16 is Drinfeld's KZ associator $\Phi_{\KZ}$, the generating series of multiple zeta values (\S1.1).
\newline 
\newline Indeed, the case of the quasi-shuffle relation is particularly significative, for the following reasons.
\newline Whereas the integral shuffle relation for $p$-adic multiple zeta values follows directly from their definition, the fact that $p$-adic multiple zeta values satisfy the quasi-shuffle relation requires a proof. However, the proof, found by Besser and Furusho \cite{Besser Furusho} and Furusho and Jafari \cite{Furusho Jafari}, is formal : it follows from formal properties of Coleman functions and does not require explicit formulas.
\newline Moreover, any explicit formula for $p$-adic multiple zeta values necessarily follows from the differential equation of the Frobenius, and thus necessarily involves the multiple harmonic sums of equation (\ref{eq:har}) and, basically, the quasi-shuffle relation is  a property of multiple harmonic sums, as we can see by the discussion in \S1.1.
\newline Finally, the regularized double shuffle relations imply conjecturally all existing algebraic relations among $p$-adic multiple zeta values. Moreover, they are both the simplest, easy to use and well-known standard family of algebraic relations among ($p$-adic) multiple zeta values.
\newline Thus, it makes sense to consider that the question of Deligne and Goncharov is a test for the relevance of an explicit computation. This is for us the second such test, the first was to shed light on Kaneko-Zagier's conjecture on finite multiple zeta values, solved in I-2.
\newline 
\newline We are going to consider in priority the quasi-shuffle relation, and solve the question of Deligne and Goncharov up to the passage between usual and "adjoint" $p$-adic multiple zeta values ; but we are also going to consider other relations.

\subsection{}

One of the main objects studied in II-1 were adjoint ($p$-adic) cyclotomic multiple zeta values, namely,
$$ \zeta_{p,\alpha}^{\Ad}\big(
\begin{array}{c} z_{j_{d}},\ldots,z_{j_{1}} \\ s_{d},\ldots,s_{1} \end{array} \big) = (-1)^{d}(\Phi_{p,\alpha}^{-1}e_{1}\Phi_{p,\alpha})[e_{0}^{s_{d}-1}e_{z_{j_{d}}} \ldots e_{0}^{s_{1}-1}e_{z_{j_{1}}}] $$
\noindent Let us explain how we are going to use them here.
\newline 
\newline In part I, we have computed $\zeta_{p,\alpha}$ by computing actually $\zeta_{p,\alpha}^{\Ad}$, and using the relation between $\zeta_{p,\alpha}$ and $\zeta_{p,\alpha}^{\Ad}$. Indeed, in the $p$-adic setting, unlike in the complex setting, adjoint multiple zeta values are as natural as multiple zeta values : adjoint $p$-adic multiple zeta values are directly related to the formulas of the Frobenius, almost more than the usual $p$-adic multiple zeta values.
\newline 
\newline Our principle for using adjoint multiple zeta values is the following : if we want to solve any question on $\zeta_{p,\alpha}$ via explicit formulas, we try to formulate an analogous question for $\zeta^{\Ad}_{p,\alpha}$, solve it for $\zeta^{\Ad}_{p,\alpha}$ and, finally, use the relation between $\zeta_{p,\alpha}$ and $\zeta^{\Ad}_{p,\alpha}$ to conclude.

\subsection{}

We are going to adopt two different point of views, and they will give us two types of results.
\newline 
\newline The first point of view relies on equation (\ref{eq:har DR RT}), and the freeness of the harmonic Ihara action, which characterize implicitly $\zeta^{\Ad}_{p,\alpha}$ in terms of multiple harmonic sums ; by studying moreover the orbit $\Orb_{\har_{\mathbb{N}}}$ of $\har_{\mathbb{N}}^{(p^{\alpha})}$ under $\circ_{\har}^{\DR,\RT}$, we will retrieve properties of $\zeta^{\Ad}_{p,\alpha}$ from properties of $\har_{\mathbb{N}}^{(p^{\alpha})}$ and $\har_{p^{\alpha}\mathbb{N}}$.
In short, in this point of view, we keep track of the fact that the Frobenius is an automorphism of the $\pi_{1}^{\un}$, and we translate in a concrete way the fact that the Frobenius is the image by a certain morphism of a point of a motivic Galois group. However, this point of view does not use explicit formulas for $\zeta^{\Ad}_{p,\alpha}$ strictly speaking, but only a certain characterization of $\zeta_{p,\alpha}$. We will obtain :
\newline 
\newline \textbf{Theorem II-2.a} Let $g \in \Ad_{\tilde{\Pi}_{1,0}(K)_{\Sigma}}(e_{1})$ and $f \in \Orb_{\har_{\mathbb{N}}}$. If $f$ and $g \circ^{\DR\text{-}\RT}_{\har} f$ satisfy the quasi-shuffle relation, then $g$ satisfies the adjoint quasi-shuffle relation from II-1.
\newline 
\newline We will also prove a converse to this implication, and some similar results concerning other families of algebraic relations. 
\newline 
\newline The second point of view relies on the explicit formula (\ref{eq:series expansion}) for $\zeta_{p,\alpha}^{\Ad}$ in terms of $\har_{p^{\alpha}}$, which was obtained in I-2 by bringing together the formulas (\ref{eq:har DR RT}) and (\ref{eq:har RT RT}). Whereas in II-1 we have transferred algebraic relations from $\Ad_{\Phi_{p,\alpha}}(e_{1})$ to $\har_{p^{\alpha}}$ along the map $\Sigma_{\inv}^{\DR}$, using equation  (\ref{eq:series expansion}), here, we are going to transfer algebraic relations in the converse way, along the map $\Sigma^{\RT}$. We will obtain :
\newline 
\newline \textbf{Theorem II-2.b} The map $\Sigma^{\RT}$ sends solutions of the (prime harmonic) quasi-shuffle relation to solutions of the adjoint quasi-shuffle relation.
\newline 
\newline We will also strengthen the view of sequences of $\har_{p^{\alpha}}$ as periods, developed in II-1 and implicit in the second point of view above, by proving that, when $N=1$ or $N$ is a prime number ($N=p'\not=p$) there exists a torsor $\text{Orb}_{\har_{q^{\mathbb{N}}}}$ under $\circ_{\har}^{\DR,\RT}$ containing the sequences prime weighted multiple harmonic sums of the form $\har_{q^{\tilde{\alpha}}}(w)$.
\newline 
\newline In order to consider the integral shuffle equation, we will define a notion of "multiple harmonic sums with reversals" which encodes the coefficients of the power series expansions of products of hyperlogarithms, and we will generalize $\circ_{\har}^{\RT}$ to multiple harmonic sums with reversals.
\newline 
\newline The Theorem II-2.a and Theorem II-2.b are two "harmonic" versions (i.e. versions involving multiple harmonic sums) of the theorem of Besser-Furusho-Jafari \cite{Besser Furusho}, \cite{Furusho Jafari}, that $p$-adic multiple zeta values satisfy the double shuffle relations. The Theorem II-2.b is an answer to the adjoint variant of the question of Deligne and Goncharov ; in a loose sense, this is also true for the Theorem II-2.a (since, strictly speaking, the Theorem II-2.a does not use explicit formulas, but a formula which is very close to explicit formulas).
\newline 
\newline In II-3, we will interpret our results from part I to this II-2 in terms of Galois theory of periods : we will formalize the idea that certain sequences of multiple harmonic sums are like periods, and that harmonic Ihara actions are like motivic Galois actions.

\subsection*{Outline}

We review the harmonic Ihara action and its properties in \S2.
\newline In \S3 and \S4 we extend our computational setting. In \S3 we define the notion of "multiple harmonic sums with reversals", and we generalize some structures of part I, in particular the $\RT$ harmonic Ihara action, to multiple harmonic sums with reversals. In \S4 we study a torsor for the harmonic Ihara actions from part I, and, when $N=1$ or $N$ is prime, we build and study another torsor which is adapted to prime weighted multiple harmonic sums of the form $\har_{q^{\tilde{\alpha}}}(w)$.
\newline The main results are proved in \S5 and \S6 : in \S5 by the properties of the harmonic torsors, and in \S6 by the comparison map $\Sigma^{\RT}$.
\newline In \S7 we prove independently some partial converses to the transfer of algebraic relations from $\zeta_{p,\alpha}$ to $\har_{p^{\alpha}}$ along the maps $\Ad(e_{1})$ and $\Sigma_{\inv}^{\DR}$ established in II-1. In \S8, we apply the problematic of reading explicitly algebraic relations to the overconvergent $p$-adic hyperlogarithms $\Li_{p,\alpha}^{\dagger}$.

\subsection*{Acknowledgments}

This work has been achieved at Universit\'{e} Paris Diderot and at Institut de Recherche Math\'{e}matique Avanc\'{e}e, Strasbourg.
It has been funded by the ERC grant n$^{o}$257638 and by the Labex IRMIA. I thank Benjamin Enriquez and Pierre Cartier for their support, and I also thank Ivan Marin for a suggestion in February of 2015.

\section{Review of the harmonic Ihara actions $\circ_{\har}$ and the comparison maps $\Sigma$ \label{Ihara}}

\numberwithin{equation}{subsection}

We review the harmonic harmonic Ihara actions defined in part I. In all this paragraph, $A$ is any complete topological $K$-algebra.

\subsection{Preliminaries on $\pi_{1}^{\un,\crys}(X_{0})$}

\subsubsection{Introduction : $\pi_{1}^{\un,\crys}(X_{0})$ viewed as $\pi_{1}^{\un,\DR}(X_{K})$ equipped with the Frobenius}

Let $X_{0} = \mathbb{P}^{1} - (\{0,\infty\} \cup \mu_{N}) \text{ }/\text{ }\mathbb{F}_{q}$, let $K=\Frac(W(\mathbb{F}_{q}))$, and $X_{K} = \mathbb{P}^{1} - (\{0,\infty\} \cup \mu_{N}) \text{ }/\text{ }K$ ; following \cite{Deligne}, \S11 and \S13, $\pi_{1}^{\un,\crys}(X_{0})$ amounts to the data of $\pi_{1}^{\un,\DR}(X_{K})$ and the Frobenius on it, where the Frobenius is an isomorphism of groupoids with connection, between $\pi_{1}^{\un,\DR}(X_{K})$ equipped with its canonical connection $\nabla_{\KZ}$ and a variant of it defined over $X_{K}^{(p)}$, where $X_{K}^{(p)} =X^{(p)} \times_{\Spec(W(k))} \Spec(K)$ and $X^{(p)}$ is the base change of $X$ by the Frobenius automorphism of $W(k)$. We will fix $\alpha \in \mathbb{N}^{\ast}$ and will consider the Frobenius of $\pi_{1}^{\un,\DR}(X_{K})$ iterated $\alpha$ times. For more details on what follows, see \S2 of II-1.

\subsubsection{$\pi_{1}^{\un,\DR}(X_{K})$}

The pro-unipotent De Rham fundamental groupoid of $X_{K}$ is a groupoid of pro-affine schemes which admits as base points the points of $X_{K}$, the non-zero tangent vectors $\vec{v}_{z}$ at $z$ to $\mathbb{P}^{1}$, and a canonical base point $\omega_{\DR}$ (the "global section" functor of the underlying Tannakian category of bundles) since $H^{1}(\overline{X_{K}},\mathcal{O}_{\overline{X_{K}}})=0$ (\cite{Deligne}, \S12).

\begin{Proposition}
i) For all points $x,y$, the scheme $\pi_{1}^{\un,\DR}(X_{K},x,y)$, has a canonical point ${}_x 1_{y}$, such that these points are compatible with the groupoid structure (for all $x,y,z$,
\newline $({}_x 1_{y}).({}_y 1_{z}) = ({}_x 1_{z})$) and induce canonical isomorphisms of schemes compatible with the groupoid structure,
$$ \pi_{1}^{\un,\DR}(X_{K},x,y) \simeq \pi_{1}^{\un,\DR}(X_{K},\omega_{\DR}) $$
\noindent ii) The pro-unipotent affine group scheme $\pi_{1}^{\un,\DR}(X_{K},\omega_{\DR})$ over $\mathbb{Q}$ is $\Spec(\mathcal{O}^{\sh,e_{Z}})$, where $\mathcal{O}^{\sh,e_{Z}}$ is the shuffle Hopf algebra over the alphabet $e_{Z} = \{e_{0},e_{z_{1}},\ldots,e_{z_{N}}\}$.
\end{Proposition}

\noindent We work via these isomorphisms and we make all the computations using $\pi_{1}^{\un,\DR}(X_{K},\omega_{\DR})$. The group $\pi_{1}^{\un,\DR}(X_{K},\omega_{\DR})(A)$ (here, $A$ can be any ring) is included functorially in the non-commutative algebra $A \langle\langle e_{Z}\rangle\rangle$ of formal power series over the variables equal to the letters of the alphabet $e_{Z}$ : the elements of $\pi_{1}^{\un,\DR}(X_{K},\omega_{\DR})(A)$ are those satisfying the "shuffle equation".
\newline 
We will use the notation $\Pi_{z,0} = \pi_{1}^{\un,\DR}(X_{K},\vec{1}_{z},\vec{1}_{0})$, for $z \in \mu_{N}(K)$.

\subsubsection{The Frobenius of $\pi_{1}^{\un,\DR}(X_{K})$}

The Frobenius iterated $\alpha \in \mathbb{N}^{\ast}$ times is characterized (see \S2 of I-1 for more details) by the couple $(\Phi_{p,\alpha},\Li_{p,\alpha}^{\dagger})$ where $\Phi_{p,\alpha}$ is the generating series of cyclotomic $p$-adic multiple zeta values and where $\Li_{p,\alpha}^{\dagger}$, reviewed in \S8, is the generating series of overconvergent $p$-adic hyperlogarithms.
\newline 
\newline We will use the Frobenius mostly through $\Phi_{p,\alpha}$, which expresses the Frobenius on the schemes $\Pi_{z,0}$, $z \in \mu_{N}(K)$. The Frobenius on $\Pi_{z,0}$ is an automorphism of the scheme $\Pi_{z,0} \times_{\Spec(\mathbb{Z})} \Spec(K)$, related to the Ihara product $\circ^{\DR} : \Pi_{z,0} \times \Pi_{z,0} \rightarrow \Pi_{z,0}$ defined as :
$$ g \circ^{\DR} f = g_{z}(e_{0},e_{z_{1}},\ldots,e_{z_{N}}).f_{z}(e_{0},g_{z_{1}}^{-1}e_{z_{1}}g_{z_{1}},\ldots,g_{z_{N}}^{-1}e_{z_{N}}g_{z_{N}}) $$
\noindent where $g_{z_{i}} = (x \mapsto z_{i}x)_{\ast}(g_{z_{N}})$ for all $i \in \{1,\ldots,N\}$.
\newline 
\newline For $\lambda \in A$, let $\tau(\lambda) : A\langle \langle e_{Z}\rangle\rangle$ be the map which multiplies the coefficient of any word $w$ by $\lambda^{\weight(w)}$.
\newline The Frobenius $\phi$ of $\Pi_{1,0}$, iterated $\alpha \in \mathbb{N}^{\ast}$ times, is characterized by 
$$ \tau(p^{\alpha}) \circ \phi^{\alpha} : f \mapsto \Phi_{p,\alpha} \circ^{\DR} f $$
\noindent where $\Phi_{p,\alpha} \in \Pi_{1,0}(K)$ is the generating series of cyclotomic $p$-adic multiple zeta values in the sense of the introduction (\S1.2).

\subsection{The harmonic Ihara action and the comparison maps}

The weight of a word over $e_{Z}$ is its number of letters, and its depth is its number of letters which are distinct from $e_{0}$.

\subsubsection{Generalities}

The following objects were defined in I-2 and I-3. 
Let $A$ be a complete topological $K$-algebra. We consider the topology on $A \langle \langle e_{Z}\rangle\rangle$ induced by the topology of convergence of functions $\{\text{words on }e_{Z}\} \rightarrow A$ which is uniform on each $\{\text{words of depth d}\}$. Let, for $z \in \mu_{N}(K)$,
$$ \tilde{\Pi}_{z,0}(A) = \{f \in \Pi_{z,0}(A)\text{ | } f[e_{z}] = f[e_{0}] = 0 \} $$
\noindent
$$ \Pi_{z,0}(A)_{\Sigma} = \big\{ f \in \Pi_{z,0}(A)\text{ | for all d }\in \mathbb{N}^{\ast}, \underset{s \rightarrow \infty}{\limsup} \{ |f[w]|_{A} \text{ | w word of weight s and depth d}\} = 0 \big\} $$
$$ \tilde{\Pi}_{z,0}(A)_{\Sigma} = \tilde{\Pi}_{z,0}(A) \cap \Pi_{z,0}(A)_{\Sigma} $$

\noindent Finally, $\tilde{\tilde{\Pi}}_{1,0}^{\RT}(A)_{\Sigma}$ is a certain subgroup of $\tilde{\Pi}_{1,0}(A)_{\Sigma}$ (defined in I-2). The harmonic Ihara actions are maps :

\begin{equation} \circ_{\har}^{\DR} : \Ad_{\tilde{\Pi}_{1,0}(A)_{\Sigma}}(e_{1}) \times 
\Map(\mathbb{N},A\langle\langle e_{Z}\rangle\rangle_{\har}^{\DR})
\rightarrow 
\Map(\mathbb{N},A\langle\langle e_{Z}\rangle\rangle_{\har}^{\DR})
\end{equation}

\begin{equation} \circ_{\har}^{\DR\text{-}\RT} : \Ad_{\tilde{\Pi}_{1,0}(A)_{\Sigma}}(e_{1}) \times 
\Map(\mathbb{N},A\langle\langle e_{Z}\rangle\rangle_{\har}^{\DR\text{-}\RT})
\rightarrow 
\Map(\mathbb{N},A\langle\langle e_{Z}\rangle\rangle_{\har}^{\DR\text{-}\RT})
\end{equation}

\begin{equation} \circ_{\har}^{\RT} : \Ad_{\tilde{\tilde{\Pi}}_{1,0}^{\RT}(A)_{\Sigma}}(e_{1}) \times 
\Map(\mathbb{N},A\langle\langle e_{Z}\rangle\rangle_{\har}^{\RT}) \rightarrow
\Map(\mathbb{N},A\langle\langle e_{Z}\rangle\rangle_{\har}^{\RT})
\end{equation}

\noindent We will not use $\circ_{\har}^{\DR}$ in the version of this text (see I-3 for its definition). The definition of $\circ_{\har}^{\RT}$ (see I-2, \S4-\S5 for details) is the lift of an elementary computation on multiple harmonic sums giving a formula for numbers $\har_{p^{\alpha}n}(w)$ in terms of the numbers $\har_{p^{\alpha}}(w')$ and $\har_{n}(w'')$, in such a way that the definition implies
$$ \har_{p^{\alpha}\mathbb{N}} = \har_{p^{\alpha}} \circ_{\har}^{\RT} \har_{\mathbb{N}}^{(p^{\alpha})} $$
\noindent Let us review the definition of $\circ_{\har}^{\DR\text{-}\RT}$. Let $A \langle\langle e_{Z} \rangle\rangle^{\lim} \subset A\langle\langle e_{Z} \rangle\rangle$, resp. $A \langle\langle e_{Z} \rangle\rangle^{\DR\text{-}\RT} \subset A\langle\langle e_{Z} \rangle\rangle$, be the vector subspace consisting of the elements $f\in A\langle\langle e_{Z} \rangle\rangle$ such that, for all words $w$ on $e_{Z}$, the sequence $(f[e_{0}^{l}w])_{l\in \mathbb{N}}$ has a limit in $A$ when $l \rightarrow \infty$, resp. is constant.
We have a map $\lim : A \langle \langle e_{Z} \rangle\rangle^{\lim} \rightarrow A \langle \langle e_{Z} \rangle\rangle^{\DR\text{-}\RT}_{\har}$ defined by, for all words $w$, $(\lim f)[w] = \displaystyle \lim_{l\rightarrow \infty} f[e_{0}^{l}w]$. Let $\circ_{\Ad}^{\DR}$ be the adjoint Ihara product (defined in I-2, \S3) which is characterized by the equation $\Ad_{g'}(e_{1}) \circ_{\Ad}^{\DR} \Ad_{g}(e_{1}) = \Ad_{g' \circ^{\DR} g }(e_{1})$. with $\tau$ as in \S2.1.2, the definition of $\circ^{\DR\text{-}\RT}_{\har}$ is :
\begin{center}
$g \circ_{\har}^{\DR,\RT} (n \mapsto h_{n}) = \big(n \mapsto \lim \big( \tau(n)(g) \circ^{\DR}_{\Ad} h_{n} \big)\big)$ 
\end{center}
\noindent We have defined in I-2 an explicit map $\Sigma^{\RT} : \Ad_{\tilde{\tilde{\Pi}}_{1,0}^{\RT}(A)_{\Sigma}}(e_{1}) \rightarrow \Ad_{\tilde{\Pi}_{1,0}(A)_{\Sigma}}(e_{1})$ which satisfies the following equation, under certain hypothesis on $g$ and $h$ :
$$ \text{ }\Sigma^{\RT}(g) \circ_{\har}^{\DR\text{-}\RT} h
= g \circ_{\har}^{\DR\text{-}\RT} h $$
\noindent and a map $\Sigma_{\inv}^{\DR} : \Ad_{\tilde{\Pi}_{1,0}(A)_{\Sigma}}(e_{1}) \rightarrow A\langle \langle e_{Z}\rangle\rangle$, such that we have proved $\Sigma^{\DR}_{\inv} \circ \Sigma^{\RT} = \id$. The formula for $\Sigma^{\RT}$ is a byproduct of the formula for $\circ_{\har}^{\RT}$. The formula for $\Sigma^{\DR}_{\inv}$ is very simple ; $\Sigma^{\DR}_{\inv}$ was used a lot in II-1.
\noindent 
In I-2, \S5.1 and \S5.2, we defined (here we assume $N=1$ for simplicity) a sequence of coefficients $\mathcal{B}_{b}^{l_{1},\ldots,l_{d}} \in \mathbb{Q}$ ($l_{1},\ldots,l_{d} \in \mathbb{Z}$, $1 \leq l \leq l_{1}+\ldots+l_{d}+d$ which appear in the formula for $\circ^{\RT}$. They are coefficients of the localized multiple harmonic sums $\frak{h}(-l_{d},\ldots,-l_{1}) : n \mapsto \sum_{0<n_{1}<\ldots<n_{d}<n} n_{1}^{l_{1}}\ldots n_{d}^{l_{d}}$. If $l_{1},\ldots,l_{d}\geq 0$, then the map $\frak{h}(-l_{d},\ldots,-l_{1})$ is polynomial and we define the coefficients $\mathcal{B}$ by the equality $\frak{h}_{n}(-l_{d},\ldots,-l_{1}) = \sum_{b=1}^{l_{1}+\ldots+l_{d}+d} \mathcal{B}_{l}^{l_{1},\ldots,l_{d}} n^{b}$ for all $n \in \mathbb{N}^{\ast}$. In the general case, $\frak{h}(-l_{d},\ldots,-l_{1})$ is a $\mathbb{Q}$-linear combination of products of polynomial functions by usual multiple harmonic sums, and we define the coefficients  $\mathcal{B}_{l}^{l_{1},\ldots,l_{d}}$ as coefficients of the purely polynomial term, i.e. the coefficients of the trivial multiple harmonic sum equal to $1$.

\subsubsection{Examples in low depth}

We review, in the case of $\mathbb{P}^{1} - \{0,1,\infty\}$ (i.e. $N=1$) and in depth one and two, the formulas for $\circ_{\har}^{\DR\text{-}\RT}$, $\circ_{\har}^{\RT}$ and $\Sigma^{\RT}$.

\begin{Examples} \label{example of RT RT har} Let $h = (h_{n})_{n\in\mathbb{N}} \in \Map(\mathbb{N},\mathbb{Q}_{p}\langle \langle e_{0},e_{1}\rangle\rangle^{\DR\text{-}\RT}_{\har})$, and $g \in \tilde{\Pi}_{1,0}(\mathbb{Q}_{p})_{\Sigma}$.
\newline i) $d=1$ : for all $s_{1} \in \mathbb{N}^{\ast}$,
\begin{equation}
(\Ad_{g}(e_{1}) \circ^{\DR,\RT}_{\har} h )(s_{1}) = \big( h_{n}(s_{1}) + \sum_{b \in \mathbb{N}}n^{s_{1}+b} \Ad_{g}(e_{1})[e_{0}^{b}e_{1} e_{0}^{s_{1}-1}e_{1}] \big)_{n\in\mathbb{N}}
\end{equation}
\noindent ii) $d=2$ : for all $s_{1},s_{2} \in \mathbb{N}^{\ast}$,
\begin{multline}
(\Ad_{g}(e_{1}) \circ_{\har}^{\DR,\RT} h )(s_{2},s_{1}) = \bigg( h_{n}(s_{2},s_{1}) + \sum_{b\in \mathbb{N}} n^{b+s_{2}+s_{1}}\Ad_{g}(e_{1})[e_{0}^{b}e_{1}e_{0}^{s_{2}-1}e_{1} e_{0}^{s_{1}-1}e_{1}] 
\\ + \sum_{r_{2}=0}^{s_{2}-1} h_{n}(s_{2}-r_{2}) n^{r_{2}+s_{1}}\Ad_{g}(e_{1})[e_{0}^{r_{2}}e_{1} e_{0}^{s_{1}-1}e_{1}] 
+ \sum_{r_{1}=0}^{s_{1}-1} h_{n}(s_{1}-r_{1}) \sum_{b\in \mathbb{N}} n^{b+s_{2}+r_{1}}\Ad_{g}(e_{1})[e_{0}^{b}e_{1}e_{0}^{s_{2}-1}e_{1} e_{0}^{r_{1}}] \bigg)_{n\in\mathbb{N}}
\end{multline}
\end{Examples}

\noindent Here we use the coefficients $\mathcal{B}$ whose definition is reviewed in \S2.2.1.

\begin{Examples} \label{example of RT RT har} Let $h = (h_{n})_{n\in\mathbb{N}} \in \Map(\mathbb{N},\mathbb{Q}_{p}\langle\langle e_{0},e_{1}\rangle\rangle^{\RT}_{\har})$, and $g \in \Ad_{\tilde{\tilde{\Pi}}_{1,0}(\mathbb{Q}_{p})_{\Sigma}}(e_{1})$.
\begin{equation}
(g \circ^{\RT}_{\har} h )_{n}(s) = h_{n}(s) + \sum_{b \geq 1} n^{b+s} \sum_{l \geq b-1} {-s \choose l} \mathcal{B}_{b}^{l} \text{ }g(s+l)
\end{equation}
\begin{multline} \label{eq: har RT RT in depth 2}
(g \circ^{\RT}_{\har} h )_{n}(s_{2},s_{1}) = 
h_{n}(s_{2},s_{1}) +
\\ \sum_{t \geq 1} n^{s_{2}+s_{1}+t} 
\sum_{l \geq t-1} \bigg[ {-s_{1} \choose l+s_{2}} \mathcal{B}_{t}^{l+s_{2},-s_{2}} - {-s_{2} \choose l+s_{1}} \mathcal{B}_{t}^{l+s_{1},-s_{1}} \bigg] g(s_{1}+s_{2}+t)  + \sum_{t \geq 1} n^{s_{1}+s_{2}+t}
	\bigg[
	\\ \sum_{\substack{l_{1},l_{2} \geq 0 \\ l_{1}+l_{2} \geq t-2}} 
	\mathcal{B}_{t}^{l_{2},l_{1}} 
	\prod_{i=1}^{2} {-s_{i} \choose l_{i}} g(s_{i}+l_{i}) +  \sum_{\substack{l_{1},l_{2} \geq 0 \\ l_{1}+l_{2} \geq t-1}}
	\mathcal{B}_{t}^{l_{1}+l_{2}} 
	\bigg( \prod_{i=1}^{2} {-s_{i} \choose l_{i}} \bigg) g(s_{2}+l_{2},s_{1}+l_{1})
	\bigg] 
	\\ - n^{s_{2}+s_{1}} \bigg[ \sum_{l_{1} \geq s_{2}-1} \mathcal{B}_{s_{2}}^{l_{1}} {-s_{1} \choose l_{1}} g(s_{1}+l_{1}) 
	- \sum_{l_{2} \geq s_{1}-1} \mathcal{B}_{s_{1}}^{l_{2}} {-s_{2} \choose l_{2}} g(s_{2}+l_{2})  \bigg] +
	\\ \sum_{\substack{ 1 \leq t < s_{2} \\ l \geq t-1}}
	n^{s_{1}+t} h_{n}(s_{2}-t) \mathcal{B}_{t}^{l} {-s_{1} \choose l} g(s_{1}+l) - \sum_{\substack{1 \leq t < s_{1} \\ l' \geq t-1}} 
	n^{s_{2}+t} h_{n}(s_{1}-t) \mathcal{B}_{t}^{l'} {-s_{2} \choose l'} g(s_{2}+l')
	\end{multline}
\end{Examples}

\begin{Examples} \label{example of Sigma RT} We have, for all $h = (h_{n})_{n\in\mathbb{N}} \in \Map(\mathbb{N},A\langle\langle e_{Z}\rangle\rangle_{\har}^{\RT})$ :
\begin{equation} \Sigma^{\RT}(h)[e_{0}^{m}e_{1}e_{0}^{s-1}e_{1}] = \sum_{l \geq m-1} {-s \choose l} \mathcal{B}_{m}^{l} h(s+l) 
\end{equation}
\begin{multline} \Sigma^{\RT}(h)[e_{0}^{m}e_{1}e_{0}^{s_{2}-1}e_{1}e_{0}^{s_{1}-1}e_{1}] = 
\sum_{\substack{l_{1},l_{2} \in \mathbb{N} \\ l_{1}+l_{2} \geq m-1 }} 
\bigg( \mathcal{B}_{m}^{l_{1}+l_{2}} h(s_{2}+l_{2},s_{1}+l_{1}) + \mathcal{B}_{m}^{l_{2},l_{1}} \prod_{i=1}^{2} h(s_{i}+l_{i}) \bigg) 
\\ + \sum_{\substack{l \geq 0 \\ l \geq m-1}} h(s_{1}+s_{2}+l) 
\bigg( {-s_{1} \choose s_{2}+l} \mathcal{B}_{m}^{l+s_{2},-s_{2}} - {-s_{2} \choose s_{1}+l} \mathcal{B}_{m}^{l+s_{1},-s_{1}} \bigg) 
\end{multline}
\end{Examples}

\section{Multiple harmonic sums with reversals and harmonic Ihara actions with reversals \label{reversals}}

Let $e_{Z}$ be the alphabet $\{e_{0},e_{z_{1}},\ldots,e_{z_{N}}\}$ which represents the sequence of differential forms $\frac{dz}{z}$, $\frac{dz}{z-z_{1}},\ldots,\frac{dz}{z-z_{N}}$. We define a notion of "words with reversals on $e_{Z}$", which extends the notion of words on $e_{Z}$ ; it indexes the generalization of multiple harmonic sums arising as the coefficients of the power series expansions of products of hyperlogarithms. We show that the notions of harmonic Ihara actions and their properties have extensions to the multiple harmonic sums with reversals.

\subsection{Words with reversals, harmonic words with reversals and multiple harmonic sums with reversals}

\subsubsection{Words with reversals}

\begin{Definition} Let the alphabet $e_{Z}^{wr}$ be the set of letters 
$\{e_{0},e_{z_{1}},\ldots,e_{z_{N}},e_{0}^{(\text{rev})}, e_{z_{1}}^{(\text{rev})},\ldots,e_{z_{N}}^{(\text{rev})}\}$.
\newline A word with reversal over $e_{Z}$ is the image of a word over $e_{Z}^{wr}$ in the quotient of $\mathbb{Q}\langle e_{Z}^{wr} \rangle$ by the ideal  $(e_{0}e_{0}^{(\rev)} - e_{0}^{(\rev)} e_{0})$. Let $\mathcal{W}(e_{Z})^{wr}$ be the set of words with reversals over $e_{Z}$.
\end{Definition}

\noindent The words with reversals over $e_{Z}$ can be represented as 
$$ e_{0}^{(s_{d}-1)+(s'_{d}-1)^{(\rev)}}e_{z_{j_{d}}}^{(x_{d})} \ldots e_{0}^{(s_{1}-1)+(s'_{1}-1)^{(\rev)}}e_{z_{j_{1}}}^{(x_{d})} e_{0}^{(s_{0}-1)+(s'_{0}-1)^{(\rev)}} $$
\noindent where $x_{d},\ldots,x_{1} \in \{\emptyset,\rev\}$, $s_{d},s'_{d},\ldots,s_{0},s_{0}' \in \mathbb{N}^{\ast}$.

\begin{Definition}
Let $K \langle \langle e_{Z}\rangle\rangle_{wr}$ be the $K$-algebra of linear forms on the $K$-vector space generated by words with reversals, viewed as formal power series, equipped with the multiplication of formal power series.
\end{Definition}

\noindent There is a bijection between the set words over $e_{Z}$ whose furthest to the right letter is not $e_{0}$ and the set of words over
$Y_{Z} = \{y_{s}^{(z_{i})} \text{ | } s \in \mathbb{N}^{\ast}, i \in \{1,\ldots,N\} \}$,
defined by $ e_{0}^{s_{d}-1}e_{z_{i_{d}}} \ldots e_{0}^{s_{d}-1}e_{z_{i_{1}}} \leftrightarrow y_{s_{d}}^{(z_{i_{d}})} \ldots y_{s_{1}}^{(z_{i_{1}})}$. It can be extended to words with reversals :

\begin{Definition} Let $Y_{Z}^{wr}$ be the alphabet of the letters 
	$\{ y_{s+s'^{(\rev)}}^{(z_{i})} \text{ | } (s,s') \in \mathbb{N}^{2} - \{(0,0)\},\text{ } i \in \{1,\ldots,N\}\}$.
\newline Let $\mathcal{W}(Y_{Z}^{wr})$ be the set of words over the alphabet $Y_{Z}^{wr}$.
\end{Definition}

\noindent Then one has a surjection from the set of words $e_{Z}$ whose furthest to the right letter is not $e_{0}$ nor $e_{0}^{(\rev)}$ to $\mathcal{W}(Y_{Z}^{wr})$, defined by $e_{0}^{s_{d}-1)+(s'_{d}-1)^{(\rev)}}e_{z_{j_{d}}}^{(x_{d})} \ldots e_{0}^{(s_{1}-1)+(s'_{1}-1)^{(\rev)}}e_{z_{j_{1}}}^{(x_{d})} \mapsto$
\newline  $y_{(s_{d}-1+\epsilon_{d})+(s'_{d}-1+\epsilon'_{d})^{(\rev)}}^{(z_{j_{d}})}\ldots y_{(s_{1}-1+\epsilon_{1})+(s_{1}-1+\epsilon'_{1})}^{(z_{j_{1}})}$ with, for all $i \in \{1,\ldots,d\}$, $(\epsilon_{i},\epsilon'_{i}) = (1,0)$ if $x_{i}=\emptyset$, and $(\epsilon_{i},\epsilon'_{i}) = (0,1)$ if $x_{i}= \rev$.
We will use often the notation 
$$ y_{s_{d}+{s'_{d}}^{(\rev)}}^{(z_{j_{d}})}\ldots y_{s_{1}+{s'_{1}}^{(\rev)}}^{(z_{j_{1}})}  \leftrightarrow \bigg( \begin{array}{cc}
z_{j_{d}},\ldots,z_{j_{1}}
\\ s_{d}+{s'_{d}}^{(\rev)},\ldots,s_{1}+{s'_{1}}^{(\rev)} \end{array} \bigg) $$

\subsubsection{Harmonic words with reversals and multiple harmonic sums with reversals}

Let us extend the notion of "harmonic words" (indices of multiple harmonic sums) to words with reversals :

\begin{Definition}
We call harmonic word with reversals over $e_{Z}$ a sequence of the form
$\bigg( \begin{array}{cc}
z_{j_{d+1}},z_{j_{d}},\ldots,z_{j_{1}}
\\ s_{d}+s_{d}'^{(\text{rev})},\ldots,s_{1}+s_{1}'^{(\text{rev})} \end{array} \bigg)$ with $s_{d},s_{d}',\ldots,s_{1},s_{1}' \in \mathbb{N}^{\ast}$ and $j_{d+1},\ldots,j_{1} \in \{1,\ldots,N\}$.
\newline We say that $\bigg( \begin{array}{cc}
z_{j_{d+1}},z_{j_{d}},\ldots,z_{j_{1}}
\\ s_{d}+s_{d}'^{(\text{rev})},\ldots,s_{1}+s_{1}'^{(\text{rev})} \end{array} \bigg)$ is of weight $\sum_{i=1}^{d} s_{i}+\sum_{i=1}^{d}s'_{i}$ and of depth $d$.
\newline Let $\mathcal{W}(e_{Z})_{\har,wr}$ be the set of harmonic words with reversals.
\end{Definition}

\begin{Definition} We call multiple harmonic sums with reversals the following numbers
$$ \frak{h}_{n} \bigg( \begin{array}{cc}
z_{j_{d+1}},z_{j_{d}},\ldots,z_{j_{1}}
\\ s_{d}+{s'_{d}}^{(\text{rev})},\ldots,s_{1}+{s'_{1}}^{(\text{rev})} \end{array} \bigg) = \sum_{0<n_{1}<\ldots<n_{d}<n} 
\frac{\big( \frac{z_{j_{2}}}{z_{j_{1}}} \big)^{n_{1}} \ldots \big(\frac{z_{j_{d+1}}}{z_{j_{d}}}\big)^{n_{d}} \big( \frac{1}{z_{j_{d+1}}}\big)^{n}}{n_{1}^{s_{1}}(n-n_{1})^{s_{1}'}\ldots n_{d}^{s_{d}}(n-n_{d})^{s_{d}'}} $$
\noindent We also call weighted multiple harmonic sums with reversals the numbers 
$\har_{n}(w) = n^{\weight(w)} \frak{h}_{n}(w)$, and prime weighted multiple harmonic sums with reversals the numbers $\har_{p^{\alpha}}(w)$, $\alpha \in \mathbb{N}^{\ast}$.
\end{Definition}

\begin{Definition} Let $K 
\langle\langle e_{Z} \rangle\rangle_{\har,wr}^{\RT}$ be the $K$-algebra of linear forms on the $K$-vector space generated by harmonic words with reversals, viewed as formal power series, equipped with the multiplication of formal power series.
\end{Definition}

\noindent For any $n \in \mathbb{N}^{\ast}$, let $\har_{n}^{wr}= (\har_{n}(w))_{w \in \mathcal{W}(e_{Z})_{\har,wr}}$ and let us view it as an element of $K 
\langle\langle e_{Z} \rangle\rangle_{\har,wr}^{\RT}$.
For any $I \subset \mathbb{N}$, let $\har_{I}^{wr} = (\har_{n}^{wr})_{n \in I}$ and let us view it as an element of $\Map(I,K 
\langle\langle e_{Z} \rangle\rangle_{\har,wr}^{\RT})$.

\subsection{Localization at $0$}

In I-2 we defined and used a notion of "localized words over $e_{Z}$", which was necessary to construct $\circ_{\har}^{\RT}$. We now extend it to a notion of "localized words with reversals over $e_{Z}$".

\subsubsection{Localized words with reversals}

\begin{Definition} A localized word with reversals over $e_{Z}$ is the image a word on the alphabet
$\{e_{0},e_{0}^{-1},e_{z_{1}},$
\newline $\ldots,e_{z_{N}},e_{0}^{(\rev)},(e_{0}^{-1})^{(\rev)},e_{z_{1}}^{(\rev)},\ldots,e_{z_{N}}^{(\rev)}\}$ in the localization of the integral non-commutative ring
\newline 
$\mathbb{Q} \langle e_{0},e_{z_{1}},\ldots,e_{z_{N}},e_{0}^{(\rev)},e_{z_{1}}^{(\rev)},\ldots,e_{z_{N}}^{(\rev)} \rangle / (e_{0}e_{0}^{(\rev)} - e_{0}^{(\rev)}e_{0})$ (where the multiplication is defined by the concatenation of words) at the multiplicative part generated by $e_{0}$ and $e_{0}^{(\rev)}$. Let $\mathcal{W}(e_{Z})_{wr,\loc}$ be the set of localized words with reversals over $e_{Z}$.
\end{Definition}

\noindent A localized word with reversals can be represented as 
$$ e_{0}^{u_{d}+{u'_{d}}^{(\rev)}}e_{z_{j_{d}}}^{(x_{d})} \ldots e_{0}^{u_{1}+{u'_{1}}^{(\rev)}}e_{z_{j_{1}}}^{(x_{d})} e_{0}^{u_{0}+{u'_{0}}^{(\rev)}} $$ 
\noindent with $u_{d},u'_{d},\ldots,u_{0},u'_{0} \in \mathbb{Z}$. 

\begin{Definition} Let $Y_{Z}^{wr,\loc}$ be the alphabet of the letters 
	$\{ y_{t+t'^{(\rev)}}^{(z_{i})} \text{ | } (t,t') \in \mathbb{Z}^{2} ,\text{ } i \in \{1,\ldots,N\}\}$.
	\newline Let $\mathcal{W}(Y_{Z}^{wr,\loc})$ be the set of words over $Y_{Z}^{wr,\loc}$.
\end{Definition}

\noindent Then one has a surjection from the set of localized words with reversals over $e_{Z}$ whose furthest to the right letter is not $e_{0}$ nor $e_{0}^{(\rev)}$ to the set of words over $Y_{Z}^{wr}$, defined by $e_{0}^{u_{d}+u_{d}^{(\rev)}}e_{z_{j_{d}}}^{(x_{d})} \ldots e_{0}^{u_{1}+u_{1}^{(\rev)}}e_{z_{j_{1}}}^{(x_{d})} \mapsto y_{(u_{d}+\epsilon_{d})+(u'_{d}+\epsilon'_{d})^{(\rev)}}^{(z_{j_{d}})}\ldots y_{(u_{1}+\epsilon_{1})+(u_{1}+\epsilon'_{1})}^{(z_{j_{1}})}$ with, for all $i \in \{1,\ldots,d\}$, $(\epsilon_{i},\epsilon'_{i}) = (1,0)$ if $x_{i}=\emptyset$, and $(\epsilon_{i},\epsilon'_{i}) = (0,1)$ if $x_{i}= \rev$. We will use often the notation
$$ y_{u_{d}+{u'_{d}}^{(\rev)}}^{(z_{j_{d}})}\ldots y_{u_{1}+{u'_{1}}^{(\rev)}}^{(z_{j_{1}})} = \bigg( \begin{array}{cc}
z_{j_{d}},\ldots,z_{j_{1}}
\\ u_{d}+{u'_{d}}^{(\rev)},\ldots,u_{1}+{u'_{1}}^{(\rev)}, \end{array}\bigg) $$

\begin{Definition} Let $K\langle \langle e_{Z}\rangle\rangle_{wr,\loc}$ be the the $K$-algebra of linear forms on the $K$-vector space generated by localized words with reversals on $e_{Z}$, viewed as formal power series, equipped with the multiplication of formal power series. 
\end{Definition}

\subsubsection{Localized harmonic words with reversals and localized multiple harmonic sums with reversals}

\begin{Definition}
We call localized harmonic word with reversals over $e_{Z}$ a sequence of the form
$\bigg( \begin{array}{cc}
z_{j_{d+1}},z_{j_{d}},\ldots,z_{j_{1}}
\\ u_{d}+u_{d}'^{(\text{rev})},\ldots,u_{1}+u_{1}'^{(\text{rev})} \end{array} \bigg)$ with $u_{d},u_{d}',\ldots,u_{1},u_{1}' \in \mathbb{Z}$ and $j_{d+1},\ldots,j_{1} \in \{1,\ldots,N\}$.
\newline We say that $\bigg( \begin{array}{cc}
z_{j_{d+1}},z_{j_{d}},\ldots,z_{j_{1}}
\\ u_{d}+u_{d}'^{(\text{rev})},\ldots,u_{1}+u_{1}'^{(\text{rev})} \end{array} \bigg)$ is of weight $\sum_{i=1}^{d} u_{i} + \sum_{i=1}^{d} u'_{i}$ and of depth $d$.
\newline Let $\mathcal{W}(e_{Z})_{\har,wr,\text{loc}}$ be the set of localized harmonic words with reversals.
\end{Definition}

\begin{Definition} We call localized multiple harmonic sums with reversals the following numbers
	$$ \frak{h}_{n} \bigg( \begin{array}{cc}
	z_{j_{d+1}},z_{j_{d}},\ldots,z_{j_{1}}
	\\ u_{d}+{u'_{d}}^{(\text{rev})},\ldots,u_{1}+{u'_{1}}^{(\text{rev})} \end{array} \bigg) = \sum_{0<n_{1}<\ldots<n_{d}<n} 
	\frac{\big( \frac{z_{j_{2}}}{z_{j_{1}}} \big)^{n_{1}} \ldots \big(\frac{z_{j_{d+1}}}{z_{j_{d}}}\big)^{n_{d}} \big( \frac{1}{z_{j_{d+1}}}\big)^{n}}{n_{1}^{s_{1}}(n-n_{1})^{u_{1}'}\ldots n_{d}^{u_{d}}(n-n_{d})^{u_{d}'}} $$
	\noindent We also call weighted localized multiple harmonic sums with reversals the numbers 
	$\har_{n}(w) = n^{\weight(w)} \frak{h}_{n}(w)$, and prime weighted localized multiple harmonic sums with reversals the numbers $\har_{p^{\alpha}}(w)$, $\alpha \in \mathbb{N}^{\ast}$.
\end{Definition}

\begin{Definition} Let $K 
	\langle\langle e_{Z} \rangle\rangle_{\har,wr,\loc}^{\RT}$ be the $K$-algebra of linear forms on the $K$-vector space generated by localized harmonic words with reversals, viewed as formal power series, equipped with the multiplication of formal power series.
\end{Definition}

\noindent For any $n \in \mathbb{N}^{\ast}$, let $\har_{n}^{wr,\loc} = (\har_{n}(w))_{w \in \mathcal{W}(e_{Z})_{\har,wr,\loc}}$ and let us view it as an element of $K 
\langle\langle e_{Z} \rangle\rangle_{\har,wr,\loc}^{\RT}$.
For any $I \subset \mathbb{N}$, let $\har_{I}^{wr,\loc} = (\har_{n}^{wr,\loc})_{n \in I}$ and let us view it as an element of $\Map(I,K 
\langle\langle e_{Z} \rangle\rangle_{\har,wr,\loc}^{\RT})$.

\subsection{Relation with iterated integrals over $\mathbb{P}^{1} - \{0,\mu_{N},\infty\}$}

We now relate the previous definitions to iterated integrals, this relation being the motivation for those definitions :

\begin{Fact} \label{fact relate reversals}(restatement of Lemma 3.13 from II-1) Let $w_{1} = \big( \begin{array}{cc} z_{j_{d}},\ldots,z_{j_{1}} \\ s_{d},\ldots,s_{1} \end{array} \big)$, $w_{2} = \big( \begin{array}{cc} z_{u_{d'}},\ldots,z_{u_{1}} \\ t_{d'},\ldots,t_{1} \end{array} \big)$ be words over $e_{Z}$, and let $n \in \mathbb{N}^{\ast}$. We have :
	$$ \big(\Li[w_{1}]\Li[w_{2}] \big)[z^{n}] = \har_{n} \bigg( \begin{array}{cc} z_{u_{1}},\ldots, z_{u_{d'}},z_{j_{d}},\ldots,z_{j_{1}} \\ t_{1}^{(\rev)},\ldots,t_{d'-1},t_{d'}^{(\rev)}+s_{d},s_{d-1},\ldots,s_{1} \end{array}  \bigg) $$
\end{Fact}

\noindent Thus, in II-1, we have encountered implicitly multiple harmonic sums with reversals, when considering the scheme $\DMR_{\Li} = \DMR_{\har_{\mathbb{N}}}$ of solutions to the double shuffle relations satisfied by multiple polylogarithms i.e. by $\har_{\mathbb{N}}$ : if we want to write an explicit definition of this scheme from the point of view of multiple harmonic sums, we have to consider multiple harmonic sums with reversals.
\newline 
\newline One can also write a localized version of Fact \ref{fact relate reversals} ; it is similar but iterated integrals are replaced by a generalization of iterated integrals whose definition involves not only the operators  $f \mapsto \int_{0} f(z)\frac{dz}{z}$ and $f \mapsto \int_{0} f(z)\frac{dz}{z-z_{i}}$, $i=1,\ldots,N$, but also the inverse of the operator $f \mapsto \int_{0} f(z)\frac{dz}{z}$, i.e. $f \mapsto z \frac{df}{dz}$.

\subsection{Prime weighted multiple harmonic sums with reversals}

This will be used in the next paragraphs :

\begin{Proposition} \label{lemma reduction of reversals prime}We have, for all localized harmonic words with reversals :  
	$$ \har_{p^{\alpha}} \bigg( \begin{array}{cc}
	z_{j_{d+1}},z_{j_{d}},\ldots,z_{j_{1}}
	\\ u_{d}+u_{d}'^{(\rev)},\ldots,u_{1}+u_{1}'^{(\rev)} \end{array} \bigg) =  \sum_{l'_{1},\ldots,l'_{d}\in \mathbb{N}} \prod_{i=1}^{d} {-u'_{i} \choose l'_{i}} \har_{p^{\alpha}} \bigg( \begin{array}{cc}
	z_{j_{d+1}},z_{j_{d}},\ldots,z_{j_{1}}
	\\ u_{d}+u_{d}'+l_{d}',\ldots,u_{1}+u_{1}'+l'_{1}\end{array} \bigg) $$
\end{Proposition}

\begin{proof} For each $n_{i} \in \{1,\ldots,p^{\alpha}-1\}$, we have $|\frac{p^{\alpha}}{n_{i}}|_{p}<1$, and thus, for any $s_{i} \in \mathbb{N}^{\ast}$, we have
	$(p^{\alpha}-n_{i})^{-u'_{i}} = (-n_{i})^{-u'_{i}}(1 - \frac{p^{\alpha}}{n_{i}})^{-u'_{i}} = n_{i}^{-u'_{i}} \sum_{l_{i}\geq 0} {-u'_{i} \choose l_{i}} \big(\frac{p^{\alpha}}{n_{i}}\big)^{l_{i}}$.
\end{proof}

\noindent Thus, the notion of multiple harmonic sums with reversals is really new only for upper bounds $n$ which are not powers of prime numbers.

\subsection{The $\RT$ harmonic Ihara action $\circ_{\har}^{\RT}$ (I-2) extended to words with reversals}

As in I-2 for $\circ_{\har}^{\RT}$, the construction of $\circ_{\har,wr}^{\RT}$ proceeds in two steps. We will build a "localized $\RT$ harmonic Ihara action with reversals" and then a map of "elimination of the localization with reversals". We leave the details of the proofs to the reader, since the proofs are similar to those of I-2. The explicit formulas for the maps will appear in the next version of this paper. Below, the subscript $\Sigma$ refers to a condition on bounds of valuations and $(K \langle\langle e_{Z} \rangle\rangle_{\har}^{\RT})_{\Sigma}$ is defined in I-2. Let $\alpha \in \mathbb{N}^{\ast}$.

\begin{Proposition-Definition} \label{prop def RT Ihara loc} There exists a map
$$ (\circ_{\har,wr}^{\RT})_{\loc} :
(K \langle\langle e_{Z} \rangle\rangle_{\har}^{\RT})_{\Sigma} \times \Map(\mathbb{N},K \langle\langle e_{Z} \rangle\rangle_{\har,wr}^{\RT}) \rightarrow
\Map(\mathbb{N},K\langle\langle e_{Z}\rangle\rangle_{\har,wr,\loc}) $$
\noindent the localized $\RT$ harmonic Ihara action with reversals (we leave the explicit formula to the reader), such that we have, (denoting by $g (\circ_{\har,wr}^{\RT})_{\loc}  f = (\circ_{\har,wr}^{\RT})_{\loc} (g, f)$)
$$ \har_{p^{\alpha}\mathbb{N},wr,\loc} = \har_{p^{\alpha}} \circ_{\har,wr,\loc}^{\RT} \har_{\mathbb{N},wr}^{(p^{\alpha})} $$
\noindent and such that the restriction of $\circ_{\har,wr,\loc}^{\RT}$ to localized words on $e_{Z}$ is equal to the map $(\circ_{\har}^{\RT})_{\loc}$ from I-2.
\end{Proposition-Definition}

\begin{proof} Similar to the construction of $(\circ_{\har}^{\RT})_{\loc}$ in I-2, \S5, via Lemma \ref{lemma reduction of reversals prime}.
\end{proof}

\begin{Proposition-Definition} \label{prop def elim} There exists an explicit map 
$$ \elim_{wr} : \Map(\mathbb{N},K \langle\langle e_{Z}\rangle\rangle_{\har,\loc}) \rightarrow \Map(\mathbb{N}, K\langle\langle e_{Z} \rangle\rangle_{\har}) $$
\noindent which we call the elimination of the localization with reversals such that :
$$ \elim_{wr} \big( \har_{\mathbb{N},wr,\loc}^{(p^{\alpha})} \big) = \har_{\mathbb{N},wr}^{(p^{\alpha})} $$
\noindent and such that the restriction of $\elim_{wr}$ to localized words on $e_{Z}$ is equal to the map $\elim$ from I-2.
\end{Proposition-Definition}

\begin{proof} Similar to the construction of the map of elimination of the localization $\elim$ in I-2, \S5.
\end{proof}

\begin{Proposition-Definition}
Let the $\RT$ harmonic Ihara action with reversals be 
$$ \circ_{\har,wr}^{\RT} = \elim_{wr} \text{ }\circ\text{ } (\circ_{\har,wr}^{\RT})_{\loc} $$
\noindent Then we have 
$$ \har_{p^{\alpha}\mathbb{N},wr} = \har_{p^{\alpha}} \circ_{\har,wr}^{\RT} \har_{\mathbb{N},wr}^{(p^{\alpha})} $$
\end{Proposition-Definition}

\begin{proof} Direct consequence of 
Proposition-Definition \ref{prop def RT Ihara loc} 
and Proposition-Definition \ref{prop def elim}.
\end{proof}

\noindent In the final version of this text, we will generalize the factorization proved in I-2
$$ "\text{ }\Sigma^{\RT}(g) \circ_{\har}^{\DR\text{-}\RT} h = g \circ_{\har}^{\RT} h\text{ }" $$
\noindent to words with reversals, with a convenient notion of $\circ_{\har,wr}^{\DR\text{-}\RT}$.

\subsection{The $\RT$ harmonic iteration of the Frobenius $\iter_{\har}^{\RT}$ (I-3) extended to words with reversals}

The explicit formula for this map will appear in the next version of this paper.

\begin{Proposition-Definition}
There exists an explicit map 
	$$ \widetilde{\iter}_{\har,\RT}^{\textbf{a},\Lambda}  : (K\langle\langle e_{Z} \rangle \rangle_{\har,wr}^{\RT})_{\Sigma} \rightarrow K[[\Lambda^{\textbf{a}}]][\textbf{a}](\Lambda)\langle\langle e_{Z} \rangle\rangle^{\RT}_{\har,wr} $$
\noindent the $\RT$ harmonic iteration of the Frobenius with reversals, such that with the map $\text{iter}_{\har,wr,\RT}^{\frac{\tilde{\alpha}}{\tilde{\alpha}_{0}},p^{\tilde{\alpha}_{0}}} : (K \langle\langle e_{Z}\rangle \rangle_{\har,wr}^{\RT})_{\Sigma} \rightarrow K \langle\langle e_{Z} \rangle\rangle^{\RT}_{\har}$ defined as the composition of $\widetilde{\text{iter}}_{\har,wr,\RT}^{\textbf{a},\Lambda}$ by the reduction modulo  $(\textbf{a} - \frac{\tilde{\alpha}}{\tilde{\alpha}_{0}},\Lambda - q^{\tilde{\alpha}_{0}})$ we have
\begin{equation} \label{eq:third of I-3}\har_{p^{\tilde{\alpha}},wr} =  \iter_{\har,\RT}^{\frac{\tilde{\alpha}}{\tilde{\alpha}_{0}},p^{\tilde{\alpha}_{0}}} (\har_{p^{\tilde{\alpha}_{0}},wr})
\end{equation}
\noindent and such that the restriction of $\iter_{\har,wr}^{\RT}$ to words on $e_{Z}$ is equal to the map $\iter_{\har}^{\RT}$ from I-3.
\end{Proposition-Definition}

\begin{proof} Similar to the construction of the map $\iter_{\har}^{\RT}$ of $\RT$ harmonic iteration (or "elevation" to a certain power $\frac{\alpha}{\alpha_{0}}$) of the Frobenius from I-3, or, alternatively, directly follows from I-3 and Proposition \ref{lemma reduction of reversals prime}.
\end{proof}

\section{Torsors for the harmonic Ihara actions}

We show certain properties of the torsor $\Orb_{\har_{\mathbb{N}}}$ for $\circ_{\har}^{\DR-\RT}$ containing $\har_{\mathbb{N}}$, and we show that, in the case of $\mathbb{P}^{1} - \{0,1,\infty\}$ and of $\mathbb{P}^{1}- \{0,\mu_{p'},\infty\}$, with $p'$ a prime number, there exists a torsor $\Orb_{\har_{q^{\mathbb{N}}}}$ for the harmonic Ihara action $\circ_{\har}^{\DR-\RT}$ containing sequences of prime weighted multiple harmonic sums whose indices are powers of $q$ (the cardinal of the base-field $\mathbb{F}_{q}$).

\subsection{The torsor $\Orb_{\har_{\mathbb{N}}}$}

\subsubsection{Definitions}

Let $A$ be a complete topological $K$-algebra. In I-2, \S3.5, we proved the following facts (with different notations) :

\begin{Definition} Let $\Orb_{\har_{\mathbb{N}}}$ be the orbit of $\har_{\mathbb{N}}^{(p^{\alpha})}$ for $\circ_{\har}^{\DR,\RT}$.
\end{Definition}

\begin{Definition} \label{def de U} Let $\mathcal{E}0^{\DR,\RT}_{\har} \subset \Map(\mathbb{N},A\langle \langle e_{Z}\rangle\rangle^{\DR-\RT}_{\har})$ be the subset of elements $h=(h_{n})_{n\in\mathbb{N}}$ such that the maps 
$n\in \mathbb{N}^{\ast} \mapsto h_{n}\big(\begin{array}{c} z_{i} \\ \emptyset,r \end{array} \big)$, $i \in \{1,\ldots,N\}$, $r \in \mathbb{N}$, are linearly independent over the ring $\frak{A}(\mathbb{Z}_{p})$ of rigid analytic functions of $n\in \mathbb{Z}_{p}$.
\end{Definition}

\begin{Proposition} \label{proposition of I-2 3.5}i) we have $\har_{\mathbb{N}}^{(p^{\alpha})} \in \mathcal{E}0^{\DR,\RT}_{\har}$,  $\mathcal{E}0^{\DR,\RT}_{\har}$ is stable by $\circ_{\har}^{\DR,\RT}$, and  $\circ_{\har}^{\DR,\RT}$ restricted to $\mathcal{E}0^{\DR,\RT}_{\har}$ is free.
\newline ii) Thus $\Orb_{\har_{\mathbb{N}}} \subset \mathcal{E}0^{\DR,\RT}_{\har}$ and $\Orb_{\har_{\mathbb{N}}}$ is a torsor for the De Rham-rational harmonic Ihara action containing $\har_{\mathbb{N}}^{(p^{\alpha})}$.
\end{Proposition}

\noindent We call $\Orb_{\har_{\mathbb{N}}}$ the $\DR-\RT$ harmonic torsor of index $\mathbb{N}$.

\subsubsection{A property of linear independence}

\begin{Definition} Let $\mathcal{E}^{\DR,\RT}_{\har} \subset \Map(\mathbb{N},A\langle \langle e_{Z}\rangle\rangle^{\DR-\RT}_{\har})$ be the subset of elements $h=(h_{n})_{n\in\mathbb{N}}$ such that the maps 
$n\in \mathbb{N}^{\ast} \mapsto \frak{h}_{n}(w)$, for $w$ of any depth, are linearly independent over the ring $\frak{A}^{\dagger}(\mathbb{Z}_{p})$ of overconvergent functions of $(n\in) \mathbb{Z}_{p}$.
\end{Definition}

\noindent We have the following result ; the i) below is a reformulation of a result of \"{U}nver :

\begin{Proposition} \label{freeness proposition}i) $\har_{\mathbb{N}^{(p^{\alpha})}} \in \mathcal{E}^{\DR,\RT}_{\har}$.
\newline ii) $\mathcal{E}^{\DR,\RT}_{\har}$ is stable by the harmonic Ihara action $\circ_{\har}^{\DR\text{-}\RT}$.
\newline iii) In particular, $\Orb_{\har_{\mathbb{N}}}^{\DR,\RT} \subset \mathcal{E}_{\har}^{\DR,\RT}$.
\end{Proposition}

\begin{proof} i) This is essentially a reformulation of a result of \"{U}nver (\cite{Unver}, Proposition 2.5). More precisely (let us take $N=1$ for simplicity), \"{U}nver shows that the sums $\displaystyle \sum_{\substack{0<n_{1}<\ldots<n_{d}<n \\ p | n_{1},\ldots,p | n_{d}}} \frac{1}{n_{1}^{s_{1}} \ldots n_{d}^{s_{d}}}$
\noindent are linearly independent over the ring of power series with coefficients in $\mathbb{Q}_{p}$ which are convergent on a closed disk around $0$ of radius $>|p|$. Replacing, in Unver's proof, the difference operator $f \mapsto (n \mapsto f(n+p) - f(n))$ by the usual difference operator $f \mapsto (n \mapsto f(n+1) - f(n))$, and $|p|$ by $1$, the same proof shows that the usual multiple harmonic sums $n \mapsto \displaystyle\sum_{0<n_{1}<\ldots<n_{d}<n } \frac{1}{n_{1}^{s_{1}} \ldots n_{d}^{s_{d}}}$ are linearly independent over the ring of the statement.
\newline ii) Similar to our proof in I-2, \S3.5 of Proposition \ref{proposition of I-2 3.5} : the action of an invertible (pro-unipotent) linear operator preserves conditions of linear independence.  iii) follows directly
\end{proof}

\begin{Remark} The Proposition \ref{freeness proposition} gives another way to prove that the equations (\ref{eq:har DR RT}) and (\ref{eq:har RT RT}) are equivalent and to arrive at the definition of the map $\Sigma^{\RT}$, which we did by a more direct method in I-2.
\end{Remark}

\begin{Remark} The Proposition \ref{freeness proposition} implies also a property of uniqueness of the map $elim$ of elimination of the localization from I-2, \S5 (which is generalized in this paper to words with reversals).
\end{Remark}

\subsection{The torsor $\Orb_{\har_{q^{\mathbb{N}}}}$ for $N \in \{1\} \cup \mathcal{P}$}

\subsubsection{Definition of $\Orb_{\har_{q^{\mathbb{N}}}}$ / another freeness property of the harmonic Ihara action}

We start with the case of $\mathbb{P}^{1} - \{0,1,\infty\}$.

\begin{Proposition} i) Assume $N=1$ or $N$ is a prime number. Let $S \subset \mathbb{N}^{\ast}$ containing an infinite sequence of numbers tending to $0$ $p$-adically. Then the part i) of Proposition \ref{proposition of I-2 3.5} remains true for the action by $\circ_{\har}^{\DR-\RT}$ on $\Map(S,A\langle\langle e_{Z}\rangle\rangle_{\har}^{\DR-\RT})$.
\newline ii) In particular the orbit $\Orb_{\har_{q^{\mathbb{N}}}}$ of the sequence $\har_{q^{\mathbb{N}}} \in \Map(q^{\mathbb{N}},K\langle\langle e_{Z}\rangle\rangle_{\har}^{\DR-\RT})$ is a torsor for $\circ_{\har}^{\DR-\RT}$.
\end{Proposition}

\begin{proof} (sketch) One can show that, our proof of I-2, \S3.5 of Proposition \ref{proposition of I-2 3.5} remains valid in this context.
\end{proof}

\begin{Definition} We call $\Orb_{\har_{q^{\mathbb{N}}}}$ the $\DR-\RT$ harmonic Ihara torsor of index $q^{\mathbb{N}}$.
\end{Definition}

\subsubsection{Analytic properties of the elements of $\Orb_{\har_{q^{\mathbb{N}}}}$}

\begin{Proposition} Assume $N=1$. The orbit of $\har_{q^{\mathbb{N}}}$ is contained in the space of analytic functions of $q^{\mathbb{N}}$ which define overconvergent analytic functions over $\mathbb{Z}_{p} \supset q^{\mathbb{N}}$.
\end{Proposition}

\begin{proof} By the main result of I-3, $\har_{q^{\mathbb{N}}}$ belongs to this space ; by the definition of $\circ_{\har}^{\DR\text{-}\RT}$, the harmonic Ihara action preserves the property of belonging to this space.
\end{proof}

\section{Reflection of algebraic relations by $\circ_{\har}^{\DR\text{-}\RT}$ in the torsors \label{torsors}}

We show compatibilities between the harmonic Ihara action and the standard algebraic relations, and prove Theorem II-2.a.

\subsection{The series shuffle relation}

We now relate the quasi-shuffle relation for  $\Phi_{p,\alpha}^{-1}e_{1}\Phi_{p,\alpha}$ (in the sense of II-1) to the (prime harmonic) quasi-shuffle relation for the numbers $h :w \mapsto \Phi_{p,\alpha}^{-1}e_{1}\Phi_{p,\alpha}[\frac{1}{1 -\Lambda e_{0}}e_{1}w]$. Let us review on an example these properties.

\begin{Example} i) In depth $(1,1)$, the quasi-shuffle relation for $\Phi_{p,\alpha}^{-1}e_{1}\Phi_{p,\alpha}$ is :
$$ (\Phi_{p,\alpha}^{-1}e_{1}\Phi_{p,\alpha})[e_{0}^{b}e_{1}e_{0}^{s_{2}-1}e_{1}e_{0}^{s_{1}-1}e_{1}\text{ }+\text{  }e_{0}^{b}e_{1}e_{0}^{s_{1}-1}e_{1}e_{0}^{s_{2}-1}e_{1}\text{ }+\text{ }e_{0}^{b}e_{1}e_{0}^{s_{2}+s_{1}-1}e_{1}] =$$
$$\sum_{\substack{b',b''\geq 0\\ b'+b''=b}}(\Phi_{p,\alpha}^{-1}e_{1}\Phi_{p,\alpha})[e_{0}^{b'}e_{1}e_{0}^{s_{1}-1}e_{1}] \times
(\Phi_{p,\alpha}^{-1}e_{1}\Phi_{p,\alpha})[e_{0}^{b''}e_{1}e_{0}^{s_{1}-1}e_{1}]$$
\noindent ii) In depth (1,1), the (prime harmonic) shuffle relation for $h$ is 
$$ h[e_{0}^{s-1}e_{1}]h[e_{0}^{t-1}e_{1}] = 
h[e_{0}^{t-1}e_{1}e_{0}^{s-1}e_{1}]
+ h[e_{0}^{s-1}e_{1}e_{0}^{t-1}e_{1}]
+ h[e_{0}^{s+t-1}e_{1}] $$
\end{Example}

\noindent Let us now write the proof of Theorem II-2.a.

\begin{proof} We write the quasi-shuffle relation for $\tilde{h}$ : for all words $w,w'$, $\tilde{h}(w)\tilde{h}(w') = \tilde{h}(w \ast w')$, where $\ast$ is the quasi-shuffle product, and we rewrite it in terms of $g$ and $h$, through the equation $\tilde{h} = g \circ_{\har}^{\DR\text{-}\RT} h$.
\newline When we encounter a product $h(z')h(z'')$, we linearize it by the quasi-shuffle relation of $h$, writing it as $h(z' \ast z'')$. It remains a linear relation between the maps $h(w) :n \in \mathbb{N} \mapsto h_{n}(w) \in \mathbb{Q}_{p}$, with coefficients in a ring of power series of $n$, expressed in terms of the coefficients of $g$. By the bounds of valuations of $g$, coming from the hypothesis $g \in \Ad_{\tilde{\Pi}_{1,0}(K)_{\Sigma}}(e_{1})$, these power series of $n$ are overconvergent over $\mathbb{Z}_{p}$.
\newline By the property of linear independence (Proposition \ref{freeness proposition}), the maps $w \mapsto h(w)$ are linearly independent over the ring of overconvergent analytic functions of $n \in \mathbb{Z}_{p}$. Thus, all the coefficients of the linear relation are $0$. The vanishing of the coefficient of $\har_{n}(\emptyset)=1$ is the adjoint quasi-shuffle relation for $\Phi_{p,\alpha}^{-1}e_{1}\Phi_{p,\alpha}$.
\end{proof}

\noindent We leave the statement and proof of the converse of this statement to the reader.

\subsection{Other relations}

In the final version of this paper we will treat the case of other algebraic relations, using a factorization of $\circ_{\har,wr}^{\RT}$ by an action  $\circ_{\har,wr}^{\DR-\RT}$.

\section{Transfer of standard algebraic relations by the comparison map ${\Sigma}^{\RT}$}

We show that the comparison map $\Sigma^{\RT}$ defined in I-2 is compatible with certain algebraic operations subjacent to the standard algebraic relations, and we prove Theorem II-2.b.

\subsection{Quasi-shuffle relation}

\subsubsection{Quasi-shuffle relation satisfied by the coefficients $\mathcal{B}$}

\begin{Lemma}
\label{lemma rough}(rough statement) The coefficients $\mathcal{B}$ (whose definition is reviewed in \S2.2.1) satisfy a form of the quasi-shuffle relation.
\end{Lemma}

\begin{proof} This follows from the fact that localized multiple harmonic sums satisfied a form of the quasi-shuffle relation (this follows directly from their definition as iterated sums), from the property of linear independence of multiple harmonic sums (Proposition \ref{freeness proposition}), and from the definition of the coefficients $\mathcal{B}$.
\end{proof}

\noindent Let us sketch the proof of Theorem II-2.b.

\begin{proof} (sketch) Essentially, Theorem II-2.b is obtained by bringing together the explicit formula for $\Sigma^{\RT}$ arising from I, the Lemma \ref{lemma rough}, and certain properties of "symmetries" of $\circ_{\har}^{\RT}$ and $\Sigma^{\RT}$ with respect to operations such as permutation of variables.
\end{proof}

\noindent In order to clarify this sketch of proof, let us give an example :

\begin{Example} i) In depth (1,1), we have the following incarnation of the quasi-shuffle relation for the coefficients of the elimination of positive powers : $ \mathcal{B}_{b}^{l_{2}+l_{1}} + \mathcal{B}_{b}^{l_{2},l_{1}} + \mathcal{B}_{b}^{l_{1},l_{2}} = \sum_{\substack{b',b'' \geq 0 \\ b'+b''=b}} \mathcal{B}_{b'}^{l_{1}}\mathcal{B}_{b''}^{l_{2}}$, for $b,l_{1},l_{2} \in \mathbb{N}$ such that $1 \leq b \leq l_{1}+l_{2}+2$. 
\newline ii) In depth (1,1), we have the following property of "symmetry" of $\circ_{\har}^{\RT}$ with respect to an exchange of variables : in the equation (\ref{eq: har RT RT in depth 2}) of Example \ref{example of RT RT har}, (which expresses the formula for $\circ_{\har}^{\RT}$ in depth $2$), certain lines of the formula are unchanged, resp. multiplied by $-1$ when we exchange $s_{2}$ and $s_{1}$.
\newline iii) Looking at the formulas for $\Sigma^{\RT}$ in low depth (Example \ref{example of Sigma RT}) and i) and ii), we see that we retrieve the adjoint quasi-shuffle relation in depth $(1,1)$ for $\Phi_{p,\alpha}^{-1}e_{1}\Phi_{p,\alpha}$ (written in Example 5.1).
\end{Example}

\subsection{Other relations}

In the next version of this paper, we will treat the case of the  symmetry equation from II-1, the shuffle relation, and the prime harmonic duality relation.

\section{Reversibility of transfers of standard algebraic relations by $\Ad(e_{1})$ and $\Sigma_{\inv}^{\DR}$}

In this section we prove, independently from the previous section, partial converse results to the transfers of standard algebraic relations along the maps $\Sigma^{\DR}_{\inv}$ and $\Ad(e_{1})$, proven in II-1. We restrict to $\mathbb{P}^{1} - \{0,1,\infty\}$ (i.e. $N=1$) for simplicity.

\subsection{Preliminary}

The following is an extension of the Definition 1.3 of II-1 to words $w$ whose furthest to the right letter is not necessarily different from $e_{0}$.

\begin{Definition} \label{def all words harmonic}Let $\Lambda$ be a formal variable. For $w$ any word on $e_{Z}$, let :
$$ \har^{\Lambda}_{\mathcal{P}^{\mathbb{N}}}(w) = \bigg( (\Phi_{p,\alpha}^{-1}e_{1}\Phi_{p,\alpha}) \big[ \frac{\Lambda^{s_{d}+\ldots+s_{1}}}{1 - \Lambda e_{0}}e_{1}w \big] \bigg)_{p \in \mathcal{P},\alpha \in \mathbb{N}^{\ast}}
$$
\end{Definition}

\subsection{For the shuffle relation}

\begin{Lemma} \label{lemma extension harmonic shuffle} Let $f \in \DMR_{0}(K)$. Then the prime harmonic shuffle relation of $\DMR_{\har_{\mathcal{P}^{\mathbb{N}}}}$ from II-1 for the numbers $h(w) = (f^{-1}e_{1}f)[\frac{\Lambda^{s_{d}+\ldots+s_{1}}}{1-\Lambda e_{0}}e_{1}w]$, namely
\begin{equation} (f^{-1}e_{1}f)\bigg[\Ad_{\widehat{\tau(\Lambda)}}\big(\frac{\Sigma^{\DR}_{\inv}}{\Lambda}\big) \bigg( (e_{0}^{s-1}e_{z_{i}}w) \text{ }\sh\text{ } w') \bigg) \bigg] = (f^{-1}e_{1}f)\bigg[\Ad_{\widehat{\tau(\Lambda)}}\big(\frac{\Sigma^{\DR}_{\inv}}{\Lambda}\big)\bigg( w \text{ }\sh\text{ } \shft_{\ast}(e_{0}^{s-1}e_{z_{j}})w') \bigg) \bigg]
\end{equation}
\noindent (see \S3.1 of II-1 for the definition of the terms of the equations) remains true for the words of Definition \ref{def all words harmonic}.
\end{Lemma}

\begin{proof} Follows directly from the proof of Proposition 4.10 in II-1, \S4.2.1.
\end{proof}

\noindent Let $\partial_{e_{1}},\tilde{\partial}_{e_{1}} : \mathcal{O}^{\sh} \rightarrow \mathcal{O}^{\sh}$ be the linear maps defined by $\partial_{e_{1}}(\emptyset)=\tilde{\partial}_{e_{1}}(\emptyset) = 0$ and by, for all words $w$,
 $\partial_{e_{1}}(e_{1}w)= w$,   
 $\partial_{e_{1}}(e_{0}w)= 0$, and
 $\tilde{\partial}_{e_{1}}(we_{1})= w$, $\tilde{\partial}_{e_{1}}(we_{0}) = 0$. Let $\mathcal{W}(\{e_{0},e_{1}\})$ be the set of words over the alphabet $\{e_{0},e_{1}\}$, which form a linear basis of $\mathcal{O}^{\sh}$.
\newline 
\newline We can now state the result of reversibility.

\begin{Proposition} \label{prop reversibility shuffle}Let $f$ be a point of $\Pi$ satisfying $f[e_{1}^{s}] = 0$ for all $s \in \mathbb{N}^{\ast}$.
\newline The following assertions are equivalent :
\newline a) $f$ satisfies the shuffle relation, i.e. $\Delta_{\sh}(f) = f \otimes f$.
\newline b) $f^{-1}e_{1}f$ satisfies the shuffle relation modulo products, i.e. $\Delta_{\sh}(f) = f \otimes 1 + 1 \otimes f$.
\newline c) The map $w \mapsto (f^{-1}e_{1}f)[\frac{1}{1-\Lambda e_{0}}e_{1}w]$, extended in the sense of Definition \ref{def all words harmonic}, satisfies the integral shuffle relation of $\DMR_{\har_{\mathcal{P}^{\mathbb{N}}}}$, extended in the sense of Lemma \ref{lemma extension harmonic shuffle}.
\newline c') The map $w \mapsto (f^{-1}e_{1}f)[\frac{1}{1-\Lambda e_{0}}e_{1}w]$ restricted to $\ker \tilde{\partial}_{e_{0}}$ satisfies the integral shuffle relation of  $\DMR_{\har_{\mathcal{P}^{\mathbb{N}}}}$ in the sense of II-1, and we have $(f^{-1}e_{1}f)[ w \text{ } \sh \text{ } e_{0} ] = 0$ for all words $w$.
\end{Proposition}

\begin{proof} Let us prove first a) $\Leftrightarrow$ b). First of all, b) is equivalent to saying that $\Delta_{\sh}(f)(f\otimes f)^{-1}$ commutes to $\Delta_{\sh}(e_{1})$. Thus the equivalence between a) and b) amounts to the following statement : let $u \in R\langle \langle e_{0},e_{1} \rangle\rangle \otimes R\langle \langle e_{0},e_{1} \rangle\rangle$ ; $u$ commutes to $\Delta_{\sh}(e_{1})$ if and only if
$u \in R\langle \langle e_{1} \rangle\rangle \otimes R\langle \langle e_{1} \rangle\rangle$.
\newline Let us prove it. For $u$ in $R\langle \langle e_{1} \rangle\rangle \otimes R\langle \langle e_{1} \rangle\rangle$, we have :
$$(\Delta_{\sh}(e_{1})u)[w\otimes w'] = u[\partial_{e_{1}}(w) \otimes w'] + u[w \otimes \partial_{e_{1}}(w')] $$
$$(u\Delta_{\sh}(e_{1}))[w\otimes w'] = u[\tilde{\partial}_{e_{1}}(w) \otimes w'] + u[w \otimes \tilde{\partial}_{e_{1}}(w')] $$
\noindent
Let $(w,w') \in \mathcal{W}(\{e_{0},e_{1}\}) \times  \mathcal{W}(\{e_{0},e_{1}\})$, with at least one among $w,w'$ not of the form $e_{1}^{l}$, $l\geq 0$ - we can assume that it is $w$ - we show that $u[w \otimes w'] = 0$.
$$ u[w \otimes w'] = u[\bar{\partial}_{e_{1}}(we_{1}) \otimes w'] $$
$$ =(u\Delta_{\sh}(e_{1}))[we_{1} \otimes w'] - u[we_{1} \otimes \tilde{\partial}_{e_{1}}(w')] = (\Delta_{\sh}(e_{1})u)[we_{1} \otimes w'] - u[we_{1} \otimes \tilde{\partial}_{e_{1}}(w')]$$
$$ = u[\partial_{e_{1}}(w)e_{1} \otimes w'] + u[we_{1} \otimes \partial_{e_{1}}(w')] - u[we_{1} \otimes \tilde{\partial}_{e_{1}}(w')] $$
\noindent Because of the hypothesis on $w$, the index of nilpotence for $\partial_{e_{i}}$ is strictly smaller for $\partial_{e_{1}}(w)e_{1}$ than for $w$. The result then follows by induction on $m+m'+m''$ where $m,n,m''$, are respectively the smallest integers satisfying :
$$ \partial_{e_{1}}^{m}(w)=0, \partial_{e_{1}}^{m'}(w')=0, (\tilde{\partial}_{e_{1}})^{m''}(w')=0 $$
\noindent Let us now prove b) $\Leftrightarrow$ c). In II-1, \S4, we have shown that, for all $w,w'$ words, and $s \in \mathbb{N}^{\ast}$, the following formal infinite sum of words :
$$ - \frac{1}{1- \Lambda e_{0}} e_{1} \bigg[ (e_{0}^{s-1}e_{1}w) \sh w' - w \sh \bigg( \frac{e_{0}^{s-1}e_{1}}{(1-\Lambda e_{0})^{s}}e_{1}w' \bigg) \bigg] $$
\noindent was a linear combination of shuffles, more precisely equal to 
$$ \sum_{s'=0}^{s-1} (e_{0}^{s'}e_{1}w) \text{ } \sh \text{ } (-1)^{s-s'} \frac{e_{0}^{s-1-s'}}{(1-\Lambda e_{0})^{s-s'}} e_{1} w' $$
\noindent Thus, the shuffle relation for $w \mapsto (f^{-1}e_{1}f) \big[\frac{1}{1-\Lambda e_{0}}e_{1}w\big]$ is true if and only if, for all $s \in \mathbb{N}$, for all $w,w'$ words, we have :
$$ (f^{-1}e_{1}f)\bigg[ \sum_{s'=0}^{s-1} (e_{0}^{s'}e_{1}w) \text{ } \sh \text{ } (-1)^{s-s'} \frac{e_{0}^{s-1-s'}}{(1-\Lambda e_{0})^{s-s'}} e_{1}w' \bigg] = 0 $$
\noindent It remains to show that these linear combinations of shuffles generate all the possible shuffles, knowing that b) is equivalent to $(f^{-1}e_{1}f)[w\text{ } \sh\text{ }w']=0$ for all non-empty words $w,w'$. For each $l \in \mathbb{N}$, the coefficient of $\Lambda^{l}$ in this linear combination is of the form :
\newline 
$( e_{0}^{s-1}e_{1}w)\text{ } \sh\text{ } (e_{0}^{l}e_{1}w') + \sum_{0\leq s'<s} c_{s'} ( e_{0}^{s'-1}e_{1}w) \text{ }\sh\text{ } (e_{0}^{l+s-s'}e_{1}w')= 0$ with $c_{s'} \in \mathbb{Q}$. By induction on the index of nilpotence of $z$ for $\partial_{e_{0}}$. This shows that for all $z,z' \in \ker \tilde{\partial}_{e_{1}}$, $z\text{ } \sh\text{ } z'$ is a linear combination of the shuffles of the statement.
\end{proof}

\subsection{Interpretation of \S7.2 : the prime harmonic De Rham pro-unipotent fundamental groupoid of $\mathbb{P}^{1} - \{0,1,\infty\}$} We use again the notation $\Pi_{1,0} = \pi_{1}^{\un,\DR}(\mathbb{P}^{1} - \{0,1,\infty\},-\vec{1}_{1},\vec{1}_{0})$.

\begin{Definition} \label{la def 1}
We denote by $\Pi_{1,0}^{\har_{\mathcal{P}^{\mathbb{N}}}} = (\Sigma_{\inv}^{\DR} \circ \Ad_{}(e_{1}))(\Pi_{1,0})$ and we call it the "prime harmonic $\Pi_{1,0}$".
\end{Definition}

\noindent Since all fibers of the groupoid $\pi_{1}^{\un,\DR}(\mathbb{P}^{1} - \{0,1,\infty\})$ are isomorphic to the same scheme $\Pi = \Spec(\mathcal{O}^{\sh,\{e_{0},e_{1}\}})$, compatibly with the groupoid structure, this definition extends to all base-points :

\begin{Definition} \label{la def 2}
We denote by $\pi_{1}^{\un,\DR_{\har_{\mathcal{P}^{\mathbb{N}}}}}(\mathbb{P}^{1} - \{0,1,\infty\}) = (\Sigma_{\inv}^{\DR} \circ \Ad(e_{1}))(\pi_{1}^{\un,\DR}(\mathbb{P}^{1} - \{0,1,\infty\}))$ and we call it the "prime harmonic  $\pi_{1}^{\un,\DR}(\mathbb{P}^{1} - \{0,1,\infty\})$".
\end{Definition}

\noindent It is a groupoid of pro-affine schemes over $\mathbb{P}^{1} - \{0,1,\infty\}$ and its tangential base-points, and is a receptacle for non-commutative generating series of sequences of prime weighted multiple harmonic sums. This makes more natural the object $\mathbb{Q}_{p}\langle \langle e_{0},e_{1}\rangle\rangle_{\har}$ defined in I-2 and I-3 and used since then.

\begin{Interpretation}
Let us recall that $\Pi_{1,0}$, and thus $\pi_{1}^{\un,\DR}(\mathbb{P}^{1} - \{0,1,\infty\})$, is defined by the shuffle equation. Thus, the Proposition \ref{prop reversibility shuffle} amounts then to saying that $\pi_{1}^{\un,\DR_{\har_{\mathcal{P}^{\mathbb{N}}}}}(\mathbb{P}^{1} - \{0,1,\infty\})$ can be defined by equations ; more precisely by the prime harmonic shuffle equation of II-1. This property is actually the main reason for stating Definition \ref{la def 1} and Definition \ref{la def 2} ; otherwise, the analogy between $\Pi_{1,0}$ and $(\Sigma_{\inv}^{\DR}\circ \Ad(e_{1}))(\Pi_{1,0})$ would be too weak to justify the notation $\Pi_{1,0}^{\har_{\mathcal{P}^{\mathbb{N}}}} = (\Sigma_{\inv}^{\DR} \circ \Ad(e_{1}))(\Pi_{1,0})$.
\end{Interpretation}

\begin{Comment} In I-2 and I-3 we have transferred the Ihara product on $\Pi_{1,0}$ along the maps $\Ad(e_{1})$ and $\Sigma_{\inv}^{\DR}$ (and this gave the adjoint Ihara product $\circ_{\Ad}^{\DR}$, and the harmonic Ihara action $\circ^{\DR}_{\har}$) ; in I-3 we have transfered $\hat{U} \Lie^{\vee}(\Pi_{1,0}) \otimes \mathbb{Q}_{p}$ along $\Sigma^{\DR}_{\inv}$, and this gave $\mathbb{Q}_{p}\langle \langle e_{0},e_{1}\rangle\rangle_{\har}^{\DR}$ ; in II-1 we have transferred algebraic relations along the maps $\Ad(e_{1})$ and $\Sigma_{\inv}^{\DR}$ ; here, we see with Definition \ref{la def 1} and Definition \ref{la def 2} that we can actually transfer the groupoid $\pi_{1}^{\un,\DR}(\mathbb{P}^{1} - \{0,1,\infty\})$ itself along the maps $\Ad(e_{1})$ and $\Sigma_{\inv}^{\DR}$. This assembles all the previous constructions in a certain sense, and thus makes them clearer.
\end{Comment}

\subsection{For the quasi-shuffle relation}

\begin{Proposition} 
i) Let $f$ a point of $\Pi$. Then $f$ satisfies the quasi-shuffle relation in depth $(1,1)$ if and only if $f^{-1}e_{1}f$ satisfies the adjoint quasi-shuffle relation of  $\DMR_{0,\Ad}$ in depth $(1,1)$.
\newline ii) Let $h \in \mathbb{Q}_{p}\langle\langle e_{0},e_{1}\rangle\rangle$. Then $h$ satisfies the adjoint quasi-shuffle relation of $\DMR_{0,\Ad}$ if and only if the map $w \mapsto h[\frac{1}{1-\Lambda e_{0}}e_{1}w]$ satisfies the prime harmonic quasi-shuffle relation of $\DMR_{\har_{\mathcal{P}^{\mathbb{N}}}}$ (which is actually equal to the usual quasi-shuffle relation : $h(w)h(w') = h(w \ast w')$ for all words $w,w'$).
\end{Proposition}

\begin{proof} i) This follows by an easy induction on $(s,t) \in (\mathbb{N}^{\ast})^{2}$ when considering words $e_{0}^{s-1}e_{1}$ and $e_{0}^{t-1}e_{1}$.
\newline ii) Clear by the definitions from II-1.
\end{proof}

\subsection{For the prime harmonic duality equation}

The reversibility of the passage between certain associator equations and certain of their "adjoint" counterparts is treated in another paper on associators \cite{J13}. We only consider the  reversibility of the passage from the adjoint setting to the "prime harmonic" setting from II-1.

\begin{Proposition} Let $f$ be a point of $\Pi$, and $e_{\infty} = - e_{0} - e_{1}$. Then $f^{-1}e_{1}f$ satisfies the equation $e_{0} + (f^{-1}e_{1}f)(e_{0},e_{1}) + (f^{-1}e_{1}f)(e_{0},e_{\infty})=0$ (the degenerated version of the equation of special automorphisms) if and only if the map $w \mapsto (-1)^{\depth(w)}(f^{-1}e_{1}f)[\frac{1}{1-\Lambda e_{0}}e_{1}w]$ satisfies the "prime harmonic" duality equation of Theorem II-1.b, namely , for all words $w$, we have :
	$$ h( w(e_{0}+e_{1},-e_{1})) = - \sum_{\substack{d'\geq 1 \\ 
			z = e_{0}^{t_{d'}-1}e_{1}\ldots e_{0}^{t_{1}-1}e_{1}}} (-1)^{d'} h(z.w) $$
\end{Proposition}

\begin{proof} Clear modulo the use of the Definition \ref{def all words harmonic} and Lemma \ref{lemma extension harmonic shuffle}, as for the case of the shuffle equation treated in \S7.2.
\end{proof}

\section{Algebraic properties of overconvergent $p$-adic hyperlogarithms $\Li_{p,\alpha}^{\dagger}$ and explicitness}

\subsection{Review of definitions}

Let $\Li_{p,\alpha}^{\dagger}$ the overconvergent variant of hyperlogarithms from part I ; they are defined as $\Li_{p,\alpha}^{\dagger}(z) = \mathcal{F}_{\ast}({}_z 1_{0})$ on the rigid analytic affinoid space $U_{0\infty}^{\an} = ( \mathbb{P}^{1,an} - \cup_{i=1}^{N}]z_{i}[ )/ K$, where $\mathcal{F}_{\ast}$ is the Frobenius of $\pi_{1}^{\un,\DR}(X_{K})$. They are non-commutative generating series of overconvergent analytic functions on  $U_{0\infty}^{\an}$.
They are characterized explicitly in terms of multiple harmonic sums and cyclotomic $p$-adic multiple zeta values by

\begin{multline} \label{eq:horizontality}
\Li_{p,\alpha}^{\dagger}(z)(e_{0},e_{z_{1}},\ldots,e_{z_{N}})\times\Li_{p,X_{K}^{(p^{\alpha})}}^{\KZ}(z^{p^{\alpha}})\big(e_{0},{\Phi^{(z_{1})}_{p,\alpha}}^{-1}e_{z_{1}}\Phi^{(z_{1})}_{p,\alpha},\ldots,{\Phi^{(z_{N})}_{p,\alpha}}^{-1}e_{z_{N}}\Phi^{(z_{N})}_{p,\alpha} \big) 
\\ = \Li_{p,X_{K}}^{\KZ}(z)(p^{\alpha}e_{0},p^{\alpha}e_{z_{1}},\ldots,p^{\alpha}e_{z_{N}})
\end{multline}

\noindent where $ \Li_{p,X_{K}}^{\KZ}$, resp. $\Li_{p,X_{K}^{(p^{\alpha})}}^{\KZ}$ is a certain solution of $\nabla_{\KZ}$ on $X_{K}$, resp. $\nabla_{\KZ}^{(p^{\alpha})}$ on $X_{K}^{(p^{\alpha})}$, whose power series expansion can be expressed in terms of multiple harmonic sum.
\newline 
\newline We stated in I-1 the following terminology and notations.

\begin{Definition} For all words $w$ and $n \in \mathbb{N}$, let $\har^{\dagger_{p,\alpha}}_{n}(w) =  n^{\weight(w)}\Li_{p,\alpha}^{\dagger}[w][z^{n}]$ (i.e. the coefficient of degree $n$ of the series expansion at $0$ of $\Li_{p,\alpha}^{\dagger}[w]$ multiplied by $n^{\weight(w)}$). We call $\har^{\dagger_{p,\alpha}}_{n}$ the multiple harmonic sums regularized by Frobenius.
\end{Definition}

\noindent In part I, we have developed the analogy between $\har_{n}^{\dagger_{p,\alpha}}$ and $\har_{n}$, which is suggested by this notation and this terminology, only on the analytic side, where both these numbers are viewed as functions of $n$ viewed as a $p$-adic integer. In II-1, we have developed an analogy between the two on the algebraic side for the restrictions to $n=p^{\alpha}$, and we have seen the close proximity between the two objects. Here, let us develop the analogy on the algebraic side, for the properties valid for all $n \in \mathbb{N}^{\ast}$.

\subsection{The shuffle relation for $\Li_{p,\alpha}^{\dagger}$ and explicitness}

By their definition, the overconvergent $p$-adic hyperlogarithms satisfy the shuffle equation : for all words $w,w'$ :
\begin{equation} \label{eq:shuffle overconvergent}
\Li_{p,\alpha}^{\dagger}[w] \Li_{p,\alpha}^{\dagger}[w'] = \Li_{p,\alpha}^{\dagger}[w\text{ }\sh\text{ }w']
\end{equation}

\noindent In other terms, since $\har_{0}^{\dagger_{p,\alpha}}=0$, this means, for all $n \in \mathbb{N}^{\ast}$ and all words $w,w'$ :
\begin{equation} 
\label{eq:shuffle overconvergent Taylor} \sum_{m=1}^{n-1}\har_{m}^{\dagger_{p,\alpha}}(w)\har_{n-m}^{\dagger_{p,\alpha}}(w') = \har_{n}^{\dagger_{p,\alpha}}(w\text{ }\sh\text{ }w')
\end{equation}

\noindent The equation (\ref{eq:shuffle overconvergent Taylor}) is the common point between $\har$ and $\har^{\dagger_{p,\alpha}}$. Following the problematic of this part II, one can ask how to retrieve (\ref{eq:shuffle overconvergent}) and (\ref{eq:shuffle overconvergent Taylor}) via explicit formulas. The problem for (\ref{eq:shuffle overconvergent}) and the problem for (\ref{eq:shuffle overconvergent Taylor}) are equivalent to each other, since we deal with analytic functions on an affinoid rigid analytic space.
\newline 
\newline If we want to do it, we must use an explicit formula for $\har_{n}^{\dagger_{p,\alpha}}$. The simplest such formula is actually the one directly given by equation (\ref{eq:horizontality}), combined to our explicit formulas for $\Phi_{p,\alpha}$ from I-2. 
\newline The Frobenius is an automorphism of the fundamental group, and we have :

\begin{Fact} If $L(z)$, $\tilde{L}(z^{p^{\alpha}})$ and $h$ are points of $\pi_{1}^{\un,\DR}(X_{K},\omega_{\DR})$, i.e. elements of $K\langle\langle e_{Z}\rangle\rangle$ satisfying the shuffle relation, if $h \in K\langle\langle e_{Z}\rangle\rangle$ satisfies the shuffle relation modulo product, and if $h^{(z_{i})} = (x \mapsto z_{i}x)_{\ast}(h)$, then 
	$$ L(z) \tilde{L}(z^{p^{\alpha}})(e_{0},{h^{(z_{1})}}^{-1}e_{z_{1}}h^{(z_{1})},\ldots,{h^{(z_{N})}}^{-1}e_{z_{N}}h^{(z_{N})})^{-1} $$
\noindent satisfies the shuffle relation.
\end{Fact}

\noindent Conversely, we showed in I-3, that for $z \in \mathbb{C}_{p}$ such that $|z|_{p} < 1$, we have $\displaystyle \tau(p^{\alpha})\Li_{p,\alpha}(z) \rightarrow_{\alpha \rightarrow \infty}\Li_{p,X_{K}}^{\KZ}(z)$ ; whence :

\begin{Fact} The shuffle equation of $\Li_{p,\alpha}^{\dagger}$ for all $\alpha \in \mathbb{N}^{\ast}$ implies, via (\ref{eq:horizontality}), the shuffle equation for $\Li_{p,X_{K}}^{\KZ}$.
\end{Fact}

\noindent \textbf{Conclusion :} the problem of the explicitness of the shuffle relation of $\Li_{p,\alpha}^{\dagger}$ is equivalent to the combination of the one of the explicitness of the shuffle relation for $\har_{n}$ and the one of the shuffle relation for $\zeta_{p,\alpha}$. We do not need to use the formulas for  $\har_{n}^{\dagger_{p,\alpha}}$ coming from part I which reflect that it is a locally analytic-exponential function of $n \in \mathbb{N}^{\ast} \subset \mathbb{Z}_{p}$.
\newline The explicitness of the shuffle relation for $\har_{n}$ is clear : its solution is Euler's proof of the shuffle equation for multiple zeta values (in depth $\leq 2$, and generalized nowadays to any depth) reviewed in II-1, \S4.4.1 ; the explicitness of the shuffle relation for $\zeta_{p,\alpha}$ is partially addressed in this paper and will be addressed more completely in the final version of this paper.

\subsection{Remarks}

In I-1, \S4.1 we have constructed an isomorphism $\comp_{\frak{A}(U_{0\infty}^{\an})}$ between the algebra of rigid analytic functions  $\frak{A}(U_{0\infty}^{\an})$ and an explicit subspace of the one of the functions $\underset{l\rightarrow\infty}{\varprojlim}{\mathbb{Z}/Np^{l}\mathbb{Z}} \rightarrow K$, where $\underset{l\rightarrow\infty}{\varprojlim}\mathbb{Z}/Np^{u}\mathbb{Z}$ is isomorphic as a topological space to the disjoint union of $N$ copies of $\mathbb{Z}_{p}$, and we have a natural inclusion $\mathbb{N}^{\ast} \subset \varprojlim \mathbb{Z}/Np^{u}\mathbb{Z}$ defined by partitioning $\mathbb{N}^{\ast}$ into the classes of congruence modulo $N$ and taking $p$-adic completions. The restriction of $\comp_{\mathfrak{A}(U_{0\infty}^{\an})}$ to the subspace of codimension $1$ of $\mathfrak{A}(U_{0\infty}^{\an})$ whose elements are the functions vanishing at $0$ was defined through
$$ (f : z \mapsto \sum_{n > 0} c_{n} z^{n}) \mapsto (c : n \in \mathbb{N}^{\ast} \mapsto c_{n} \in K) $$
\noindent where the right-hand side above extended in a canonical way a function of $\varprojlim \mathbb{Z}/Np^{u}\mathbb{Z}$ by a property of continuity. The inverse $\comp_{\mathfrak{A}(U_{0\infty}^{\an})}^{-1}$
of $\comp_{\mathfrak{A}(U_{0\infty}^{\an})}$ has also a simple explicit formula (this formula is a priori particular to the example of $U_{0\infty}^{\an}$ ; I-1, \S4.1). This makes more precise the equivalence between the problem of reading explicitly (\ref{eq:shuffle overconvergent}) and (\ref{eq:shuffle overconvergent Taylor}).
\newline 
\newline We wonder whereas the shuffle relation of $\har^{\dagger_{p,\alpha}}$ can be read via its expression as a locally analytic function on $\varprojlim \mathbb{Z}/Np^{u}\mathbb{Z}$, in a way which shows how it is related to the shuffle relation for $\har_{p^{\alpha}}$. It seems that this is possible by using the proof of Theorem I-2.b, which gives an indirect way to retrieve the local analyticity of $\har^{\dagger_{p,\alpha}}$.
\newline When we take $z=\infty$ in (\ref{eq:shuffle overconvergent}), we obtain
\begin{equation} \label{eq:final eq} \zeta_{p,\alpha}^{(\infty)}(w)\zeta_{p,\alpha}^{(\infty)}(w') = \zeta_{p,\alpha}^{(\infty)}(w\text{ }\sh\text{ }w') 
\end{equation}
\noindent where $\zeta_{p,\alpha}^{(\infty)}$ is the analogue of $\zeta_{p,\alpha}$ defined with base-points at $(\vec{1}_{\infty},\vec{1}_{0})$. The shuffle relation of $\zeta_{p,\alpha}^{(\infty)}$ follows from the one of $\zeta_{p,\alpha}$, by writing the generating series  $\Phi_{p,\alpha}^{(\infty)}$ of  $\zeta_{p,\alpha}^{(\infty)}$ in terms of $\Phi_{p,\alpha}$ : 
$\Phi_{p,\alpha}^{(\infty)} = (z \mapsto \frac{1}{z})_{\ast}(\Phi_{p,\alpha})^{-1} \times \Phi_{p,\alpha}$ (I-1, \S5.2). However, we also wonder whereas equation (\ref{eq:final eq}) can be understood via the formula for $\har^{\dagger_{p,\alpha}}$ as a locally analytic function on $\varprojlim \mathbb{Z}/Np^{u}\mathbb{Z}$.

\end{document}